%% file: artin.tex
\documentclass[10pt]{amsart}

\usepackage{amssymb,verbatim,url}
\allowdisplaybreaks

\usepackage{mathtools}

\makeatletter
\DeclareRobustCommand\widecheck[1]{{\mathpalette\@widecheck{#1}}}
\def\@widecheck#1#2{%
    \setbox\z@\hbox{\m@th$#1#2$}%
    \setbox\tw@\hbox{\m@th$#1%
       \widehat{%
          \vrule\@width\z@\@height\ht\z@
          \vrule\@height\z@\@width\wd\z@}$}%
    \dp\tw@-\ht\z@
    \@tempdima\ht\z@ \advance\@tempdima2\ht\tw@ \divide\@tempdima\thr@@
    \setbox\tw@\hbox{%
       \raise\@tempdima\hbox{\scalebox{1}[-1]{\lower\@tempdima\box
\tw@}}}%
    {\ooalign{\box\tw@ \cr \box\z@}}}
\makeatother

\newcommand{\twgood}{\text{\Smiley{}}}
\newcommand{\twbad}{\text{\Frowny{}}}
\newcommand{\twnull}{\phantom{\text{\Smiley{}}}}
\renewcommand{\twgood}{\bullet}
\renewcommand{\twbad}{\circ}
\renewcommand{\twnull}{\phantom{\text{$\bullet$}}}

\newcommand{\nc}{\newcommand}
\nc{\bl}[1]{{{\color{blue}#1}}}
\nc{\br}[1]{{{\color{brown}#1}}}
\nc{\ch}{\bl{\checkmark}}
\nc{\fix}[1]{\mathcal{F}(#1)}
\nc{\fixx}{\mathcal{F}}
\nc{\cG}{\mathcal{G}}
\nc{\cH}{\mathcal{H}}
\nc{\C}{\mathbf{C}}
\nc{\G}{{\mathbb{G}}}
\nc{\Q}{\mathbf{Q}}
\nc{\Z}{\mathbf{Z}}
\nc{\F}{\mathbf{F}}
\nc{\R}{\mathbf{R}}
\nc{\Qp}{\mathbf{Q}_p}
\nc{\gen}[1]{\langle #1\rangle}
\nc{\ds}{\displaystyle}
\nc{\Oc}{\mathcal{O}}
\nc{\OF}{\Oc_F}
\nc{\OK}{\Oc_K}
\nc{\OL}{\Oc_L}
\nc{\Fr}{\mbox{Fr}}
\nc{\cI}{\mathcal{I}}
\nc{\cP}{\mathcal{P}}
\nc{\cC}{\mathcal{C}}
\nc{\cZ}{\mathcal{Z}}
\nc{\cL}{\mathcal{L}}
\nc{\cK}{\mathcal{K}}
\nc{\OLg}{\Oc_{L^g}}
\nc{\fP}{\mathfrak{P}}
\nc{\fp}{\mathfrak{p}}
\nc{\fd}{\mathfrak{d}}
\nc{\ft}{\mathfrak{t}}
\nc{\fc}{\mathfrak{c}}
\nc{\g}{{\rm gal}}
\nc{\ta}{{\,t}}
\nc{\GL}{\mathrm{GL}}
\nc{\PGL}{\mathrm{PGL}}
\nc{\PSL}{\mathrm{PSL}}
\nc{\SL}{\mathrm{SL}}
\nc{\bbar}[1]{\overline{#1}}
\nc{\vd}{\vec{d}\, }
\nc{\into}{\hookrightarrow}
\nc{\cots}{, \ldots, }
\nc{\pots}{+ \cdots +}
\nc{\seq}{\subseteq}
\nc{\note}[1]{\begin{center}\fbox{\textbf{#1}}\end{center}}
\nc{\cmmt}[1]{}
\nc{\bp}{\bullet \!\!\!\! +} 
\nc{\sst}{\scriptstyle}
\nc{\Imnaught}{[-\widecheck{\chi},\widehat{\chi}]}
\nc{\alp}{\underline{\alpha}}
\nc{\what}{\widehat}
\nc{\walp}{\what{\alp}}
\nc{\Chi}{\mathcal{X}}
\nc{\lChi}{\!\raisebox{-0.2em}{\tiny$\Chi$}}

\nc{\marp}{\marginpar}
\nc{\jj}[1]{\marp{\bl{#1}}}
\nc{\dr}[1]{\marp{\br{#1}}}

\DeclareMathOperator{\gal}{Gal}
\DeclareMathOperator{\Gal}{Gal}
\DeclareMathOperator{\Out}{Out}

\newtheorem{thm}{Theorem}[section]

\newtheorem{lemma}[thm]{Lemma}
\newtheorem{cor}[thm]{Corollary}

\numberwithin{table}{section}
\numberwithin{figure}{section}
\numberwithin{equation}{section}

\theoremstyle{remark}

\title{Artin $L$-functions of small conductor}

\author{John W.\ Jones}
\address{School of Mathematical and Statistical Sciences, Arizona
  State University, PO Box 871804, Tempe, AZ 85287} 
\email{jj@asu.edu}

\author{David P.\ Roberts}
\address{Division of Science and Mathematics, University of
  Minnesota-Morris, Morris, MN 56267}
\email{roberts@morris.umn.edu}

\thanks{DPR's work on this paper was supported by Grant \#209472 from the Simons Foundation 
and Grant \#1601350 from the National Science Foundation.}

\begin{document}
\begin{abstract}  
 We study the problem of finding the Artin $L$-functions with the
 smallest conductor for a given Galois type.  We adapt standard
  analytic techniques to our novel situation of fixed Galois type and
  get much improved lower bounds on the smallest conductor.  For small
  Galois types we use complete tables of number fields to determine
  the actual smallest conductor.
\end{abstract}

\maketitle

\setcounter{tocdepth}{1}
\tableofcontents

\section{Overview}
\label{overview}
   Artin $L$-functions $L(\Chi,s)$ are remarkable analytic objects built from number fields.  
Let $\overline{\Q}$ be the algebraic closure of the 
rational number 
 field $\Q$ inside the field of complex numbers $\C$.   
Then Artin $L$-functions are indexed by continuous complex characters 
$\Chi$ of the absolute Galois group 
$\G = \Gal(\overline{\Q}/\Q)$, with the unital character $1$ giving the Riemann zeta
function $L(1,s) = \zeta(s)$.   An important problem in modern number theory is
to obtain a fuller understanding of these higher
analogs of the Riemann zeta function.   The analogy is expected to be very tight:  all Artin $L$-functions are 
expected by the Artin conjecture to be entire except perhaps for a
pole at $s=1$; they  
are all expected to satisfy the Riemann hypothesis that
all zeros with $\mbox{Re}(s) \in (0,1)$ satisfy $\mbox{Re}(s)=1/2$.  
 
   The two most basic invariants of an Artin $L$-function $L(\Chi,s)$ 
 are defined via the two explicit elements of $\G$, the
 identity $e$
 and the complex conjugation element $\sigma$.   
 These invariants are the degree $n = \Chi(e)$ and the
 signature $r = \Chi(\sigma)$ respectively.  
 A measure of the complexity
 of $L(\Chi,s)$ is its conductor $D \in \Z_{\geq 1}$,
 which can be computed from the discriminants of 
related number fields. 
    It is best for purposes such as ours to focus instead
    on the root conductor $\delta = D^{1/n}$.

     In this paper, we aim to find the simplest Artin 
$L$-functions exhibiting a given Galois-theoretic 
behavior.  To be more precise, 
 consider triples $(G,c,\chi)$ consisting of a finite group $G$,
 an involution $c \in G$, and a faithful character $\chi$.  We say that $\Chi$
 has {\em Galois type} $(G,c,\chi)$ if there is a surjection $h : \G \rightarrow
 G$ with $h(\sigma) = c$, and 
 $\Chi = \chi \circ h$.   Let $\cL(G,c,\chi)$ be the
 set of $L$-functions of type $(G,c,\chi)$, and
 let $\cL(G,c,\chi;B)$ be the subset
 consisting of $L$-functions with root conductor
 at most $B$.  Two natural problems for any given Galois type $(G,c,\chi)$ are
 \begin{itemize}
 \item[\textbf{1}:]  Use known and the above
 conjectured properties of $L$-functions to get a 
lower bound $\frak{d}(G,c,\chi)$ on the root conductors 
 of $L$-functions in $\cL(G,c,\chi)$.
 \item[\textbf{2}:] Explicitly
 identify the sets $\cL(G,c,\chi;B)$ with $B$ as large as possible.
 \end{itemize}    
 This paper gives answers to both problems, although for
 brevity we often fix only $(G,\chi)$ and work instead
 with the sets $\cL(G,\chi;B) := \cup_c \cL(G,c,\chi;B)$.

 There is a large literature on a special case of the situation 
 we study.  Namely let $(G,c,\phi)$ be a Galois type
 where $\phi$ is the character of a transitive permutation
 representation of $G$.   Then the  set 
 $\cL(G,c,\phi;B)$ is exactly 
 the set of Dedekind zeta functions $\zeta(K,s)$ 
 arising from a corresponding
 set $\cK(G,c,\phi;B)$ of arithmetic equivalence classes of 
 number fields.    In this context, root conductors   
 are just root discriminants, and lower bounds date back to
 Minkowski's work on the geometry of numbers.   
 Use of Dedekind zeta functions as in {\bf 1} above began with
 work of Odlyzko \cite{od-disc1,od-disc1a,od-disc2}, Serre
 \cite{serre-minorations}, Poitou
 \cite{poitou-minorations,poitou-petits}, and Martinet \cite{martinet}. 
Extensive responses to {\bf 2} came shortly 
 thereafter, with papers often focusing on a single degree
 $n=\phi(e)$.  Early 
 work for  quartics, quintics, sextics,
 and septics include respectively
 \cite{bf-quartic2,f-quartic1,bfp-quartic3}, 
 \cite{spd}, \cite{pohst,BMO,olivier1,olivier2,olivier3}, and
 \cite{letard}.
 Further results towards {\bf 2} in higher degrees are
 extractable from the websites 
 associated to \cite{jr-global-database},
 \cite{kluners-malle}, and \cite{LMFDB}.
 
 The full situation that we are studying here was identified
 clearly by Odlyzko in \cite{odlyzko-durham}, who responded to {\bf 1} with 
 a general lower bound.  However this more general
 case of Artin $L$-functions has almost no subsequent presence
 in the literature.  A noticeable exception is a 
 recent paper of Pizarro-Madariaga \cite{PM}, who improved 
 on Odlyzko's results on {\bf 1}.  A novelty of our 
 paper is the separation into Galois types.  For many 
 Galois types this separation allows us to go 
 considerably further on {\bf 1}.  This paper is also the first 
 systematic study of {\bf 2} beyond the case 
 of number fields.   
 
Sections~\ref{Artin} and \ref{Signature} review background on Artin
$L$-functions and tools used to bound their conductors.
Sections~\ref{Type}--\ref{otherchoices} form the new material on the
lower bound problem {\bf 1}, while Sections~\ref{S5}--\ref{discussion}
focus on the tabulation problem {\bf 2}.  Finally, Section~\ref{asymp}
returns to {\bf 1} and considers asymptotic lower bounds for
root conductors of Artin $L$-functions in certain families.   
In regard to {\bf 1}, Figure~\ref{amalia-plot} and
Corollary~\ref{limitcor} give a quick indication of how 
our type-based lower bounds compare with the earlier
degree-based lower bounds.  In regard to both {\bf 1} and {\bf 2}, 
Tables~\ref{tablelabel1}--\ref{tablelabel8}  
show how the new lower bounds compare with
actual first conductors for many types.
 
\cmmt{ 
  $**********$
 Section~\ref{Artin} reviews some of the basic
 formalism behind Artin $L$-functions,
 with emphasis on the
 direct connection with
 number fields.    In particular,
 suppose $\phi$ is the 
 character of a transitive
 permutation representation of $\G$;  
 then the  set 
 $\cL(G,c,\phi;B)$ 
 is naturally identified
 with the set of Dedekind zeta functions $\zeta(K,s)$ 
 arising from a corresponding
 set $\cK(G,c,\phi;B)$ of arithmetic equivalence classes of 
 number fields.

 Section~\ref{Signature} summarizes the
 literature on lower bounds
 for root conductors.
 We restrict attention to
  bounds which are 
 conditional on standard analytic hypotheses,
 namely the Artin conjecture and 
 the Riemann hypothesis for the relevant
 auxiliary $L$-functions.  
  If $\chi$ takes only 
 nonnegative values, 
 like permutation characters do,
  there 
 is a direct method for
 obtaining a conditional lower bound
 for the smallest root conductor. 
 For general $\chi$, there is 
 an indirect method 
 that uses the very general
 conductor relation \eqref{gencondrel} 
 based on the auxiliary character $\phi_S = \chi \bar{\chi}$.     
 
 Sections~\ref{Type}--\ref{otherchoices} form 
 the new material on analytic lower
 bounds.  Instead
 of the general
 conductor relation involving $\phi_S$,
 we use type-based conductor
 relations that depend on
 the choice of an auxiliary nonnegative character $\phi$.  
 The form
 of these relations is given in \eqref{maincondrel} and
 the relation is made more explicit in Section~\ref{choices}
 for four simple and useful types of  $\phi$.
 Section~\ref{otherchoices} describes the polytope of possible choices for $\phi$.   
 
 Our formalism of triples $(G,c,\chi)$ captures the strong tradition in 
 the literature of paying close attention to the placement of complex conjugation.  
 However in the remaining sections we keep things relatively brief by
 working simply with $(G,\chi)$, which we 
 also call a type.   We pursue the problem of identifying 
 the sets  $\cL(G,\chi;B) := \cup_c \cL(G,c,\chi;B)$ with
 a focus on finding at least the the minimal root conductor 
 $\delta_1(G,\chi)$.

 Sections~\ref{S5}--\ref{discussion} focus on this tabulation problem.    
 Conductor relations similar 
 to \eqref{maincondrel} let one
 use known typically large complete
 lists  coming from number field tables
 to determine new typically smaller complete lists
 of $L$-functions.  
 Section~\ref{S5} explains 
 the process and illustrates it
 by working out the case $G=S_5$ 
 in detail.   Section~\ref{Tables} 
 considers $G$ arising as 
 transitive subgroups of 
 $S_n$, 
 restricting to $n \leq 9$ for
  solvable groups and $n \leq 7$ for 
  nonsolvable groups.  
  It presents conditional 
  lower bounds 
  $\mathfrak{d}(G,\chi)$
  and then initial segments 
 $\cL(G,\chi;B)$,
 almost always non-empty.  
 As illustrated by Figure~\ref{amalia-plot}, our 
 type-based lower bounds $\mathfrak{d}(G,\chi)$ are 
 usually substantially larger then
 the best degree-based lower bounds coming
 from \cite{PM}.   We find always 
 \begin{equation}
 \mathfrak{d}(G,\chi)< \delta_1(G,\chi),
 \end{equation}
 thus no contradiction to the Artin
 conjecture or Riemann hypothesis.  
 In fact, typically $\delta_1(G,\chi)$ 
 is considerably larger than 
 $\mathfrak{d}(G,\chi)$.   
 Section~\ref{discussion} 
 discusses several aspects of
 the information presented in the
 tables.   
 
 Section~\ref{asymp} shows in Corollary~\ref{limitcor}
 that restricting the type in simple ways gives
 much increased
 asymptotic lower bounds.  It speculates that these larger
 lower bounds may hold even with no restriction
 on the type.  
 }
 
 Artin $L$-functions
 have recently become much more
 computationally accessible through
 a package implemented in {\em Magma}
 by Tim Dokchitser.   Thousands
 are now collected 
 in a section on the LMFDB \cite{LMFDB}.  
 The present work 
 increases our understanding of all this
 information in several ways, including by providing completeness
 certificates for certain ranges.

\section{Artin $L$-functions} 
\label{Artin}
 In this section we provide some background.   An important 
 point is that our problems allow us to restrict consideration
 to Artin characters which take rational values only.  In this setting,
 Artin $L$-functions can be expressed as products and quotients 
 of roots of Dedekind zeta functions, minimizing the background needed.
General references on Artin $L$-functions include \cite{Mar77,Mur01}. 

\subsection{Number fields}  
A number field $K$ has many invariants
relevant for our study.   First of all, there is the degree $n = [K:\Q]$.  
The other invariants we need are local in that they are associated 
with a place $v$ of $\Q$ and can be read off from the corresponding
completed algebra $K_v = K \otimes \Q_v$, but not from other
completions.  For $v=\infty$, the complete 
invariant is the signature $r$, defined by $K_\infty \cong \R^{r} \times \C^{(n-r)/2}$.
It is more convenient sometimes to work with the eigenspace dimensions
for complex conjugation, $a = (n+r)/2$ and $b = (n-r)/2$.  
For an ultrametric place
$v=p$, the full list of invariants is complicated.   The most
basic one is the positive integer $D_p = p^{c_p}$ generating the discriminant
ideal of $K_p/\Q_p$.   We package the $D_p$ into the single
invariant $D = \prod_p D_p \in \Z_{\geq 1}$, the absolute  discriminant
of $K$.  

\subsection{Dedekind zeta functions}  Associated with
a number field is its Dedekind zeta function
\begin{equation}
\label{prodser}
\zeta(K,s) = \prod_{p} \frac{1}{P_p(p^{-s})} = \sum_{m=1}^\infty \frac{a_m}{m^s}.
\end{equation}
Here the polynomial $P_p(x) \in \Z[x]$ is a $p$-adic invariant.  It has
degree $\leq n$ with equality if and only if $D_p=1$.  The integer
 $a_m$ is the number of ideals of index $m$ in the ring of integers $\OK$.

 \subsection{Analytic properties of Dedekind zeta functions} 
 Let $\Gamma_\R(s) = \pi^{-s/2} \Gamma(s/2)$, where 
 $\Gamma(s) = \int_0^\infty x^{s-1} e^{-x} dx$ is the standard
 gamma function. Let
 \begin{equation}
 \label{completed}
 \what{\zeta}(K,s) = D^{s/2} \Gamma_\R\left(s\right)^a \Gamma_\R\left(s+1\right)^b \zeta(K,s).
 \end{equation}
Then this completed Dedekind zeta function 
 $\what{\zeta}(K,s)$ meromorphically continues to the whole complex plane, 
 with simple poles at $s=0$ and $s=1$ being its only singularities.
It satisfies the functional equation
$\what{\zeta}(K,s) = \what{\zeta}(K,1-s)$.
 
\subsection{Permutation characters}  We recall from the introduction that throughout this paper we are taking 
  $\overline{\Q}$ to be the algebraic closure of $\Q$ in $\C$ and 
  $\G = \Gal(\overline{\Q}/\Q)$ its absolute Galois group.     A degree $n$ number field
  $K$ then corresponds to the transitive $n$-element 
  $\G$-set $X = \mbox{Hom}(K,\overline{\Q})$.  
  A number field thus has a permutation character $\Phi = \Phi_K = \Phi_X$
  with $\Phi(e)=n$.   Also signature has the character-theoretic
  interpretation $\Phi(\sigma) = r$, where $\sigma$
  as before is the complex conjugation element.
    
  \subsection{General characters and Artin $L$-functions}  Let $\Chi$ be a character of $\G$.  
  Then one has an associated Artin $L$-function $L(\Chi,s)$ and conductor
  $D_\Chi$, agreeing
  with the Dedekind zeta function $\zeta(K,s)$ and the discriminant $D_K$
   if $\Chi$ is the permutation
  character of $K$.    The function $L(\Chi,s)$ has both an Euler product and
  Dirichlet series representation as in \eqref{prodser}. 
     In general, if  
  $\Phi = \sum_\Chi m_{\lChi} \Chi$ then   
\begin{align}
\label{decomp}
L(\Phi,s) & = \prod_{\Chi} L(\Chi,s)^{m_{\lChi}}  & D_\Phi & = \prod_{\Chi} D_\Chi^{m_{\lChi}}.
\end{align}
One is often interested in \eqref{decomp} where the $\Chi$ are irreducible characters.  

For a finite set of primes $S$, let $\overline{\Q}_S$  
  be the compositum of all number fields in $\overline{\Q}$
  with discriminant divisible only by primes in $S$.  
  Let $\G_S = \Gal(\overline{\Q}_S/\Q)$ be the corresponding
  quotient of $\G$.  Then for primes $p \not \in S$ one has
  a well-defined Frobenius conjugacy class $\Fr_p$ in $\G_S$.
  The local factor $P_p(x)$ in \eqref{prodser} is the characteristic polynomial
  $\det(1 - \rho(\Fr_p) x)$, where $\rho$ is a representation
  with character $\Chi$.

   \subsection{Relations with other objects} 
   Artin $L$-functions of degree $1$ are exactly Dirichlet
 $L$-functions, so that $\Chi$ can be identified with a
 faithful character of the quotient group $(\Z/D\Z)^\times$
 of $\G$, with $D$ the conductor of $\Chi$.
  Artin $L$-functions coming from irreducible degree $2$ characters
   and conductor $D$ are expected to come from
   cuspidal  modular forms on $\Gamma_1(D)$, holomorphic if $r=0$ and
    nonholomorphic otherwise.  This expectation
    is proved in all cases, except for those with $r = \pm 2$ 
    and projective image the nonsolvable group $A_5$. 
      In general,  to understand how an Artin $L$-function $L(\Chi,s)$
    qualitatively relates to other objects, one needs to
    understand its Galois theory, including 
    the placement of complex conjugation; in other words, 
    one needs to identify its Galois type.   
    To be more quantitative, one brings in the conductor.  
  
  \subsection{Analytic Properties of Artin $L$-functions} 
  An Artin $L$-function
  has a meromorphic continuation and functional
  equation, although each with an extra complication 
  in comparison with the special case of Dedekind zeta functions.  
   For the 
  meromorphic continuation, the behavior at $s=1$ is known:
  the pole order is the multiplicity $(1,\Chi)$ of $1$ 
  in $\Chi$.    The complication is that one has 
  poor 
  control over other possible poles.  The Artin conjecture  for $\Chi$
  says however that there are no poles other than $s=1$.

The completed $L$-function
\[ \what{L}(\Chi,s) = D_\Chi^{s/2} \Gamma_\R(s)^a \Gamma_\R(s+1)^b L(\Chi,s)\]
satisfies the functional equation
   \[
  \what{L}(\Chi,1-s) = w \what{L}(\overline{\Chi},s)
  \]
with root number $w$.
  Irreducible characters of any compact group come
  in three types, orthogonal, non-real, and symplectic.
  The type is identified by the Frobenius-Schur
  indicator, calculated with respect to the Haar probability
  measure $dg$:
  \[
  FS(\chi)  = 
  \int_{G} \chi(g^2) \, dg \in \{-1,0,1\}.
  \]
  For orthogonal characters $\Chi$ of $\G$, one has $\Chi = \overline{\Chi}$
  and moreover $w = 1$.   The complication in
  comparison with permutation characters 
  is that for the other two types, 
  the root number $w$ is 
  not necessarily $1$.  For symplectic characters,
  $\Chi = \overline{\Chi}$ and $w$
  can be either of the two possibilities $1$ or
  $-1$.  For non-real characters, $\Chi \neq \overline{\Chi}$ 
  and $w$ is some algebraic number of norm $1$.
  
  Recall from the introduction that an Artin $L$-function is said to satisfy the Riemann 
  hypothesis if all its zeros in the critical strip $0<\mbox{Re}(s)<1$ are actually on the critical line $\mbox{Re}(s)= 1/2$.  
  We will be using the Riemann hypothesis through 
 Lemma~\ref{lowboundlem}.  If we replaced the function \eqref{Odlyzko} with the appropriately 
 scaled version of (5.17) from \cite{PM}, 
 then our lower bounds would be only conditional on the Artin
 conjecture, which is   
 completely known for some Galois types $(G,c,\chi)$.  However 
 the bounds obtained would be much smaller, and the comparison
 with first conductors as presented in Tables~\ref{tablelabel1}--\ref{tablelabel8} below would be 
 less interesting.  

\subsection{Rational characters and rational Artin $L$-functions}   
\label{rat-chars}
The abelianized Galois 
group $\G^{\rm ab}$ acts on continuous complex characters of profinite
groups through its action  
on their values.    If $\Chi'$ and $\Chi''$ are conjugate via this action 
then their conductors agree: 
\begin{equation}
\label{discequal}
D_{\Chi'} = D_{\Chi''}.  
\end{equation}
Our study is 
simplified by this equality because it allows us to 
study a given irreducible character $\Chi'$ by
studying instead the rational character $\Chi$
obtained by summing its conjugates.    

By the Artin induction theorem \cite[Prop.~13.2]{feit}, 
a rational character $\Chi$
can be expressed as a rational linear combination of 
permutation characters:
\begin{equation}
\label{chiexpress}
\Chi = \sum k_\Phi \Phi.
\end{equation}
For general characters $\Chi'$, computing the Frobenius traces 
$a_p = \Chi'(\Fr_p)$ requires the results of \cite{dok-dok}. 
Similarly the computation of bad Euler factors and the root number $w$ present difficulties.  
For Frobenius traces and bad Euler factors,
these complications are 
not present for rational characters $\Chi$ because of \eqref{chiexpress}.

\section{Signature-based analytic lower bounds}   
\label{Signature}
Here and in the 
 next section we aim to be brief, with the main point being to explain how type-based
 lower bounds are usually much larger than signature-based lower bounds.  
We employ the standard framework for establishing lower bounds for
conductors and discriminants, namely Weil's explicit formula.
General references for the material
 in this section are \cite{odlyzko-durham,PM}.

\subsection{Basic quantities}      
\label{basic-quantities}

The theory allows general test functions that satisfy certain axioms.
We work only with a function introduced by Odlyzko (see \cite[(9)]{poitou-petits}),  
\begin{equation}
\label{Odlyzko}
f(x) = \left\{ \begin{array}{ll}  {\displaystyle (1-x) \cos(\pi x) + \frac{\sin(\pi x)}{\pi}}, & \mbox{ if $0 \leq x \leq 1$,} \\
0, & \mbox{ if $x>1$}.
\end{array}
\right.
\end{equation}
For $z \in [0,\infty)$ let
\begin{align*}
N(z) & =  \log(\pi) +  \int_0^{\infty}
\frac{e^{-x/4}+e^{-3x/4}}{2(1-e^{-x})} f(x/(2z))
  - \frac{e^{-x}}{x}  \,dx, \\
&= \gamma+ \log(8\pi) + \int_0^\infty \frac{f(x/z)-1}{2\sinh(x/2)} \, dx \\ 
&= \gamma+\log(8\pi) + \int_0^z \frac{f(x/z)-1}{2\sinh(x/2)} \, dx
-\log\left(\frac{e^{z/2}+1}{e^{z/2}-1} \right), \\
R(z) & =   \int_0^{\infty}\frac{e^{-x/4}-e^{-3x/4}}{2(1-e^{-x})} f(x/(2z))\, dx, \\
&= \int_0^z \frac{f(x/z)}{2\cosh(x/2)}\, dx, \\
P(z) & =  4 \int_0^{\infty}  f(x/z) \cosh(x/2)\, dx \\
&= \frac{256 \pi^2 z \cosh^2(z/4)}{(z^2+4\pi^2)^2}.
\end{align*}
The simplifications in the integrals for $N(z)$ and $R(z)$ are fairly
standard and apply to
any test function, with the exception of the final steps which make
use of the support for $f(x)$.  Evaluation of $P(z)$ depends on the
choice of $f(x)$.
The integrals for $N(z)$ and $R(z)$ cannot be evaluated in closed form
like the third, but, as indicated in \cite[\S2]{poitou-petits}, they do have simple limits $N(\infty) = \log(8 \pi) +  
\gamma$ and  
$R(\infty) = \pi/2$ as $z \rightarrow \infty$.  
Here $\gamma \approx 0.5772$ is the Euler $\gamma$-constant.
The constants $\Omega = e^{N(\infty)} \approx 44.7632$ and
$e^{R(\infty)} \approx 4.8105$,
as well as their product $\Theta = e^{N(\infty)+R(\infty)} \approx
215.3325$, will play important 
roles in the sequel.   

\subsection{The quantity $M(n,r,u)$}  Consider triples $(n,r,u)$ 
of real numbers with $n$ and $u$ positive and $r \in [-n,n]$.  For such
a triple, define
\[
M(n,r,u) =  \mbox{Max}_z \left( \exp \left(  N(z) + \frac{r}{n} R(z)  - \frac{u}{n} P(z) \right) \right).
\]
It is clear that $M(n,r,u) = M(n/u,r/u,1)$.  
Accordingly we regard $u=1$ as the essential case
and abbreviate $M(n,r)=M(n,r,1)$.   
For fixed $\epsilon \in [0,1]$ and $u>0$, one has the asymptotic 
behavior
\begin{equation}
\label{asymptotic1}
\lim_{n \rightarrow \infty} M(n,\epsilon n) = \Omega^{1-\epsilon} \Theta^{\epsilon} \approx 
44.7632^{1-\epsilon} 215.3325^{\epsilon}.
\end{equation}
Figure~\ref{contourM} gives one a feel for
the fundamental function $M(n,r)$.   
Particularly important are the univariate functions 
$ 
M(n,0)$ and $
 M(n,n)$ corresponding
to the left and right edges of this figure.   
\begin{figure}[htb]
\centering
\includegraphics[width=4.5in]{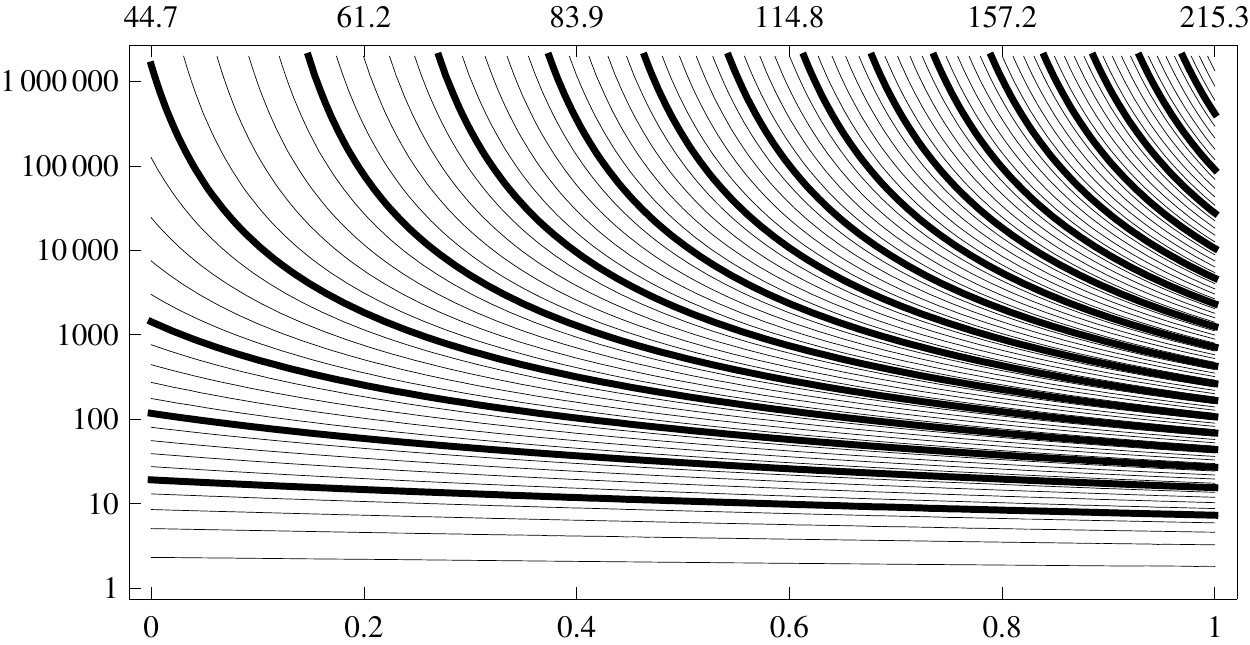}
\caption{\label{contourM} A contour plot of $M(n,\epsilon n)$ in the
 window $[0,1] \times [1,1 \, 000 \, 000]$ 
of the $\epsilon$-$n$ plane,
with a vertical logarithmic scale and contours at $2$, $4$, $6$, $8$, ${\bf 10}$, \dots, {\bf 170}, 172, 174, 176.  
 Some limits for $n \rightarrow \infty$   
are shown on the upper boundary.}
\end{figure}

\subsection{Lower bounds for root discriminants} Suppose that 
 $\Phi$ is a nonzero Artin character which takes real values only.  We say 
 that $\Phi$ is nonnegative if 
\begin{equation}
\label{nonnegativity}
\Phi(g) \geq 0 \mbox{ for all $g \in \G$.}
\end{equation}
This nonnegativity ensures that the inner product 
$(\Phi,1)$ of $\Phi$ with the unital character $1$ is positive.  
A central result of the theory, a special case of the statement in \cite[(7)]{poitou-petits}, then serves us as a lemma. 
\begin{lemma}  
\label{lowboundlem}
The lower bound 
\[
\delta_{\Phi} \geq M(n,r,u).
\]
is valid for all nonnegative characters $\Phi$ with
$(\Phi(e),\Phi(\sigma),(\Phi,1)) = (n,r,u)$
and $L(\Phi,s)$ satisfying the Artin conjecture and the Riemann hypothesis. 
\end{lemma}
If $\Phi$ is a permutation character, then the nonnegativity condition \eqref{nonnegativity}
is automatically satisfied.  This makes the application of the analytic theory 
to lower bounds of root discriminants of fields relatively straightforward.  

\subsection{Lower bounds for general Artin conductors}
To pass from nonnegative characters to general characters, the classical 
method uses the following lemma.
\begin{lemma}[Odlyzko \cite{odlyzko-durham}]
\label{lemma1} The conductor relation
\begin{equation}
\label{gencondrel}
\delta_\Chi \geq \delta_{\Phi}^{n/(2n-2)}
\end{equation}
holds for any degree $n$ character 
$\Chi$ and its absolute square  $\Phi = |\Chi|^2$.
\end{lemma}
\noindent A proof of this lemma from first principles is given in \cite{odlyzko-durham}.

Combining Lemma~\ref{lowboundlem} with Lemma~\ref{lemma1} one gets 
the following immediate consequence

\begin{thm} The lower bound
\label{thm1}
\[
\delta_\Chi \geq M(n^2,r^2,w)^{n/(2n-2)}
\]
is valid for all characters $\Chi$ with 
$(\Chi(e),\Chi(\sigma),(\Chi,\overline{\Chi})) = (n,r,w)$ such that $L(|\Chi|^2,s)$ 
satisfies the Artin conjecture and the Riemann hypothesis.  
\end{thm}
This theorem 
is essentially the main result
in the literature on lower bounds for Artin conductors.  It appears in
\cite{odlyzko-durham, PM} with the right side replaced by 
explicit bounds.
For fixed $\epsilon \in [-1,1]$ and $w>0$, one has the asymptotic 
behavior
\begin{equation}
\label{asymptotic2}
\lim_{n \rightarrow \infty} M(n^2,\epsilon^2 n^2,w) = \Omega^{(1-\epsilon^2)/2} \Theta^{\epsilon^2/2} \approx 
6.6905^{1-\epsilon^2} 14.6742^{\epsilon^2}.
\end{equation}
The bases $\Omega \approx 44.7632$ and $\Theta \approx 215.3325$ of \eqref{asymptotic1} serve as limiting
lower bounds for root discriminants via Lemma~\ref{lowboundlem}.   However it is only their square roots
$\sqrt{\Omega} \approx 6.6905 $ and $\sqrt{\Theta} \approx 14.6742$
which Theorem~\ref{thm1} gives as limiting lower bounds for root conductors.  This
 discrepancy will be addressed in Section~\ref{asymp}.

\section{Type-based analytic lower bounds}
\label{Type}
     In this section we establish Theorem~\ref{thm2}, which is 
  a family of lower bounds on the root conductor
  $\delta_\Chi$ of a given Artin character, dependent
  on the choice of an auxiliary character $\phi$.

\subsection{Conductor relations}
Let $G$ be a finite group, $c$ an involution in $G$, $\chi$ a faithful
character of $G$, and $\phi$ 
a non-zero real-valued character of $G$.   
Say that a pair of  Artin characters 
 ($\Chi$,$\Phi$) has joint
type $(G,c,\chi,\phi)$ if there is a surjection $h : \G \rightarrow G$ with
$h(\sigma)=c$,  $\Chi = \chi \circ h$, and $\Phi = \phi \circ h$.  
  
Write the conductors respectively as
\begin{align*}
D_\Chi & = \prod_p p^{c_p(\Chi)}, & 
D_\Phi & = \prod_p p^{c_p(\Phi)}.
\end{align*}
Just as in the last section, we need lower bounds on $D_\Chi$ in terms of  
$D_{\Phi}$.   Our paper \cite{jr-tame-wild} produces bounds
of this sort in the context of many characters.
Here we present some of these results restricted to the setting
of  two characters, but otherwise following the notation of
\cite{jr-tame-wild}.

For $\tau \in G$, let 
$\bar{\tau}$ be its order. Let $\psi$ be a rational character of $G$.    
Define two similar numbers, 
\begin{align}
\label{twosimilar}
\what{c}_\tau(\psi)&= \psi(e)-\psi(\tau), &
c_\tau(\psi)& =  
 \psi(e) - \frac{1}{\bar{\tau}} 
\sum_{k|\bar{\tau}} \varphi(\bar\tau/k) \psi(\tau^k). 
\end{align}
Here $\varphi$ is the Euler totient function given by $\varphi(k) = |(\Z/k)^\times|$. 
For the identity element $e$, one clearly has
$\what{c}_e(\psi) = c_e(\psi)=0$.    
When $\bar{\tau}$ is prime, the functions on rational characters
defined in \eqref{twosimilar} are proportional:   
$(\bar{\tau}-1) \what{c}_\tau(\psi) =\bar{\tau}{c}_\tau(\psi)$.  

The functions $\what{c}_\tau$  and ${c}_\tau$ are related to ramification as follows.  
Let $\Psi$ be an Artin character corresponding to $\psi$ under $h$.   
If $\Psi$ is tame at $p$ then
\begin{equation}
\label{tameidentity}
c_p(\Psi) = c_\tau(\psi),
\end{equation}
for $\tau$ corresponding to a generator of tame inertia.  The 
identity \eqref{tameidentity}  holds because $c_\tau(\psi)$ is the
number of non-unital eigenvalues
of $\rho(\tau)$ for a representation $\rho$ with character $\psi$.  
For general $\Psi$, there is
a canonical expansion
\begin{equation}
\label{wildbound}
c_p(\Psi) = \sum_{\tau} k_\tau \what{c}_\tau(\psi),  
\end{equation}
with always $k_\tau \geq 0$,  
coming from the filtration by higher
ramification groups on the $p$-adic inertial subgroup of $G$.  

Because \eqref{twosimilar}--\eqref{wildbound} are 
only correct for $\psi$ rational, when we apply them to characters $\chi$ 
and $\phi$ of interest, we are always assuming that $\chi$ and $\phi$
are rational.   As explained in \S\ref{rat-chars}, the restriction to rational
characters still allows obtaining general lower bounds.  Also, as will 
be illustrated by an example in \S\ref{spectralwidth}, focusing on rational
characters does not reduce the quality of these lower bounds.  

For the lower bounds we need, we define the parallel quantities
\begin{align}
\label{alphaprod}
\what{\alpha}(G,\chi,\phi) & = \min_{\tau \in G-\{e\}} \frac{\what{c}_\tau(\chi)}{\what{c}_\tau(\phi)}, & 
\alpha(G,\chi,\phi) & = \min_{\tau \in G-\{e\}} \frac{c_\tau(\chi)}{c_\tau(\phi)}. 
\end{align}
Let $B(G,\chi,\phi)$ be the best lower bound, valid for all primes $p$, that one can make on $c_p(\Chi)/c_p(\Phi)$
by purely local arguments.   As emphasized in
\cite[\S2]{jr-tame-wild}, $B(G,\chi,\phi)$ can in   
principle be calculated by individually inspecting all possible 
$p$-adic ramification behaviors.  The above discussion says
\begin{equation}
\label{aba}
\what{\alpha}(G,\chi,\phi) \leq B(G,\chi,\phi) \leq \alpha(G,\chi,\phi).
\end{equation}
The left inequality holds because of the nonnegativity of the $k_\tau$ in
 \eqref{wildbound}.  The right inequality holds because of \eqref{tameidentity}.  

  A central theme of \cite{jr-tame-wild}
 is that one is often but not always 
 in the extreme situation 
  \begin{equation}
  \label{ba}
 B(G,\chi,\phi) = \alpha(G,\chi,\phi).
 \end{equation}
 For example, it often occurs in practice that 
 the minimum in the expression \eqref{alphaprod} for $\what{\alpha}(G,\chi,\phi)$
 occurs at a $\tau$ of prime order.  Then the proportionality
 remark above shows that in fact
 all three quantities in \eqref{aba} are 
 the same, and so in particular \eqref{ba} holds.   
  As a quite different example, Theorem~7.3 of \cite{jr-tame-wild} 
 says that if $\phi$ is the regular character of $G$ and $\chi$ is a
 permutation character, then
 \eqref{ba} holds.   Some other examples of \eqref{ba} are worked out in 
 \cite{jr-tame-wild} by explicit analysis of wild ramification;
 a few examples show that  $B(G,\chi,\phi) < \alpha(G,\chi,\phi)$ is possible too.

\subsection{Root conductor relations} To switch the focus from conductors to root conductors,
we multiply all three quantities in \eqref{aba} by $\phi(e)/\chi(e)$ to obtain 
\begin{equation}
\label{aba2}
\walp(G,\chi,\phi) \leq b(G,\chi,\phi) \leq \alp(G,\chi,\phi).
\end{equation}
Here the elementary purely group-theoretic quantity $\walp(G,\chi,\phi)$ is improved
to the best bound $b(G,\chi,\phi)$ which in turn often agrees
with a second more complicated but still purely group-theoretic quantity $\alp(G,\chi,\phi)$.
The notations $\what{\alpha}$, $\alpha$, $\what{\underline{\alpha}}$, 
$\underline{\alpha}$ are all taken from Section~7
of \cite{jr-tame-wild} while
the notations $B$ and $b$ correspond to quantities not named
in \cite{jr-tame-wild}.

Our discussion establishes the following lemma.  
\begin{lemma} 
\label{lemma2} 
The conductor relation
\begin{equation}
\label{maincondrel}
\delta_\Chi \geq \delta_\Phi^{b(G,\chi,\phi)} 
\end{equation}
holds for all pairs of Artin characters $(\Chi,\Phi)$ with joint type  
of the form  $(G,c,\chi,\phi)$. 
\end{lemma}

\subsection{Bounds via an auxiliary Artin character $\Phi$.} 
\label{bounds}
For $u \in \{\walp,b,\alp\}$, define
\begin{equation}
\label{mdef}
m(G,c,\chi,\phi,u) = M(\phi(e),\phi(c),(\phi,1))^{u(G,\chi,\phi)}.
\end{equation}
Just like Lemma~\ref{lowboundlem} combined with Lemma~\ref{lemma1} to give Theorem~\ref{thm1},
so too Lemma~\ref{lowboundlem} combines with Lemma~\ref{lemma2} to give 
the following theorem.
\begin{thm} \label{thm2}  The lower bound 
\begin{equation}
\label{thm2bound}
\delta_\Chi \geq m(G,c,\chi,\phi,b)
\end{equation}
is valid for all character pairs $(\Chi,\Phi)$ of joint type $(G,c,\chi,\phi)$ 
such that $\Phi$ is non-negative and $L(\Phi,s)$ satisfies the Artin conjecture and the Riemann hypothesis.
\end{thm}
Computing the right side of \eqref{thm2bound} is difficult because the 
base in \eqref{mdef} requires evaluating the maximum of a complicated
function, while the exponent $b(G,\chi,\phi)$ involves an exhaustive
study of wild ramification.  Almost always in the sequel, $\chi$ and $\phi$ are 
rational-valued and we replace $b(G,\chi,\phi)$ by $\walp(G,\chi,\phi)$; in the common case
that all three quantities of \eqref{aba2} are equal, this is no loss.

\section{Four choices for $\phi$} 
\label{choices}
This section fixes a type $(G,c,\chi)$ where the faithful character $\chi$ 
is rational-valued and uses the notation
$(n,r) = (\chi(e),\chi(c))$.  The section introduces four nonnegative characters $\phi_i$ built from 
$(G,\chi)$.  For the first character $\phi_L$, it makes $m(G,c,\chi,\phi_L,b)$, the lower
bound appearing in Theorem~\ref{thm2}, more explicit.  For the
remaining three characters $\phi_i$, it makes the perhaps smaller quantity
$m(G,c,\chi,\phi_i,\walp)$ more explicit.

    Two simple quantities enter into the constructions as follows.
Let $X$ be the set of 
values of $\chi$, so that $X \subset \Z$ by our rationality
assumption.  Let $-\widecheck{\chi}$ be the least element of $X$.  The greatest element of
$X$ is of course $\chi(e)=n$, and we let $\widehat{\chi}$ be the second greatest element.  
Thus, $-\widecheck{\chi}  <0  \leq \widehat{\chi} \leq n-1$.  

\subsection{Linear auxiliary character}  
A simple nonnegative character associated to $\chi$ is 
$\phi_L = \chi+\widecheck{\chi}$.    Both
$m(G,c,\chi,\phi_L,\walp)$ and $m(G,c,\chi,\phi_L,\alp)$ 
easily evaluate to 
\begin{equation}
\label{formlinear}
m(G,c,\chi,\phi_L,b) =  M(n+\widecheck{\chi},r+\widecheck{\chi},\widecheck{\chi})^{(n + \widecheck{\chi})/n}.
\end{equation}
The character $\phi_L$ seems most promising as an auxiliary character 
when $\widecheck{\chi}$ is very small.  

In \cite[\S3]{PM}
the auxiliary character $\chi+n$ is used, which has
the advantage of being nonnegative for any rational
character $\chi$.  Odlyzko also uses $\chi+n$ in
\cite{odlyzko-durham}, and suggests using the auxiliary character
$\phi_L = \chi+\widecheck{\chi}$ since it 
gives a better bound whenever
$\widecheck{\chi}<n$.  
This strict inequality holds exactly when the center of $G$ has odd
order.

\subsection{Square auxiliary character} 
Another very simple nonnegative character
associated to $\chi$ is $\phi_S = \chi^2$.  This character gives  
\begin{equation}
\label{formsquare}
m(G,c,\chi,\phi_S,\walp) = M(n^2,r^2,(\chi,\chi))^{n/(n+\widehat{\chi})}.
\end{equation}
The derivation of \eqref{formsquare} uses the simple formula in
\eqref{twosimilar} for $\what{c}_\tau$.
The formula for $c_\tau$ in \eqref{twosimilar} is more complicated, and we
do not expect a simple general formula for $m(G,c,\chi,\phi_S,{\alp})$,
nor for the
best bound $m(G,c,\chi,\phi_S,b)$ in Theorem~\ref{thm2}.  

The character $\phi_S$ is used prominently in \cite{odlyzko-durham, PM}.
When $\widehat{\chi}=n-2$,
the lower bound $m(G,c,\chi,\phi_S,\walp)$ coincides with that of
Lemma~\ref{lemma1}.
Thus for $\widehat{\chi}=n-2$,  
Theorem~\ref{thm2} with $\phi = \phi_S$ gives the same bound
as Theorem~\ref{thm1}.  
On the other hand, as soon as $\widehat{\chi}<n-2$, Theorem~\ref{thm2} 
with $\phi = \phi_S$ is stronger. 
The remaining case $\widehat{\chi}=n-1$ occurs only three times among
the $195$ characters we consider in Section~\ref{Tables}.  In these 
three cases, the bound in Theorem~\ref{thm1} is stronger because the exponent 
is larger.
However, in each of these cases, the tame-wild principle applies 
\cite{jr-tame-wild} and we can use 
exponent $m(G,c,\chi,\phi_S,{\alp})$, which gives the same bound
as Theorem~\ref{thm1} in two cases, and a better bound in the third.

\subsection{Quadratic auxiliary character}  Let $-\widetilde{\chi}$ be the greatest negative element of the
set $X$ of values of $\chi$.   A modification of the given character $\chi$ is $\chi^* = \chi+\widetilde{\chi}$,
with degree $n^* = n+\widetilde{\chi}$ and signature $r^* = r + \widetilde{\chi}$.   A modification of 
$\phi_S$ is $\phi_Q = \chi \chi^*$.    The function $\phi_Q$ takes only nonnegative values
because the interval $(-\tilde{\chi},0)$ in the $x$-line where $x(x+\tilde{\chi})$ is negative is disjoint
from the set $X$ of values of $\chi$.   The lower bound associated to $\phi_Q$ is 
\begin{equation}
\label{formquad}
m(G,c,\chi,\phi_Q,\walp) =  M(nn^*,rr^*,(\chi,\chi))^{(n^*)/(n^*+\widehat{\chi})}.
\end{equation}
Comparing formulas \eqref{formsquare} and \eqref{formquad}, $n^2$ strictly increases to $nn^*$ and 
$n/(n+\widehat{\chi})$ increases to $n^*/(n^* + \widehat{\chi})$.   
In the totally real case $n=r$, the monotonicity of the function $M(n,n)$ as exhibited in the right edge of Figure~\ref{contourM}
then implies that $m(G,c,\chi,\phi_S,\walp)$
strictly increases to $m(G,c,\chi,\phi_Q,\walp)$.   Even outside the totally 
real setting, one can expect that $\phi_Q$ almost always yields a
better lower bound than $\phi_S$.   
The character $\phi_Q$ seems promising as an auxiliary character when
$\widehat{\chi}$ is very small so that the
exponent is near $1$ rather than its lower limit of $1/2$.  
As for the square case, we do not expect a simple
formula for the best bound $m(G,c,\chi,\phi_Q,b)$ in Theorem~\ref{thm2}.

\subsection{Galois auxiliary character}   Finally there is a strong candidate for
a good auxiliary character that does not depend on $\chi$, 
namely the regular character $\phi_G$.  By definition,
$\phi_G(e) = |G|$ and else $\phi_G(g)=0$.   In this case one has
\begin{equation}
\label{formula3a}
m(G,c,\chi,\phi_G,\walp) = M(|G|,\delta_{ce}|G|,1)^{(n-\widehat{\chi})/n}.
\end{equation}
Here $\delta_{ce}$ is defined to be $1$ in the totally real case and $0$ otherwise.
This auxiliary character again seems most promising when $\widehat{\chi}$ is small. 
As in the square and quadratic cases, we do not expect a simple formula 
for $m(G,c,\chi,\phi_G,b)$.

\subsection{Spectral bounds and rationality} 
\label{spectralwidth}  To get large lower bounds on root conductors,
one wants $\widecheck{\chi}/n$ to be small for \eqref{formlinear} or
$\widehat{\chi}/n$ to be small for
\eqref{formsquare}--\eqref{formula3a}.   The analogous quantities
$\widecheck{\chi}_1/n_1$ 
 and $\widehat{\chi}/n_1$ are 
well-defined for a general real character $\chi_1$, and replacing  
$\chi_1$ by the sum $\chi$ of its conjugates can
substantially reduce them.  

For example, let $p$ be a prime
congruent to 1 modulo 4, and let $G$ be the simple group $\PSL_2(p)$.
Then $G$ has two irrational irreducible characters, say $\chi_1$ and $\chi_2$, 
both of degree $(p+1)/2$.  For each, its set of values is 
$$\left\{\frac{-\sqrt{p}-1}{2},-1,0,1,\frac{\sqrt{p}-1}{2},
\frac{p+1}{2}\right\}$$
(except that 
$1$ is missing if $p=5$).   However 
for $\chi = \chi_1+\chi_2$, the set of values is just 
$\{-2,0,2,p+1\}$.     Thus in passing from $\widecheck{\chi}_1/n_1$ to
$\widecheck{\chi}/n$, one saves a factor of $\sqrt{p}+1$.  Similarly
in passing from $\widehat{\chi}_1/n_1$ to
$\widehat{\chi}/n$, one saves a factor of $\sqrt{p}-1$.

\section{Other choices for $\phi$} 
\label{otherchoices} To apply Theorem~\ref{thm2}
for a given Galois type $(G,c,\chi)$, one needs to choose an auxiliary character
$\phi$.   We presented four choices in Section~\ref{choices}.   We discuss
all possible  choices here, using $G=A_4$ and $G=A_5$ as illustrative 
examples.  

\subsection{Rational character tables} 
As a preliminary, we review the notion of rational character table.
Let $G^\sharp = \{C_j\}_{j \in J}$ be the set of 
power-conjugacy classes in $G$.  Let $G^{\rm rat} = \{\chi_i\}_{i \in I}$ be the set of rationally
irreducible characters.   These sets have the same size $k$ and one
has a $k\times k$ matrix
$\chi_i(C_j)$, called the rational character table.  

\begin{table}[htb]
\[
\begin{array}{c|rrr  c  c|rrrr}
  A_4         &1A & 2A & 3AB  & \;\;\;\;\;\;\;\;\;\;\;\;\;\;\;\;\;\;  & A_5 &  1A & 2A & 3A & 5AB \\
\cline{1-4} \cline{6-10}
\chi_1 & 1 & 1 & 1 &  & \chi_1 & 1 & 1 & 1 & 1 \\
\chi_2 & 2 & 2  & -1 &     &\chi_4 & 4 & 0 & 1 & - 1 \\
\chi_3 & 3 & - 1& 0 && \chi_5 & 5 & 1 & -1 & 0 \\
          \multicolumn{2}{c}{\;}        &      &&&\chi_6 & 6 & -2 & 0 & 1  \\
\end{array}
\]
\caption{\label{ratchartables} Rational character tables for $A_4$ and $A_5$} 
 \end{table}

Two examples are given in Table~\ref{ratchartables}.
 We index characters by their degree, with
$I = \{1,2,3\}$ for $A_4$ and $I = \{1,4,5,6\}$ for $A_5$.     All characters are absolutely irreducible except
for $\chi_2$ and $\chi_{6}$, which each break as a sum of two conjugate irreducible complex characters.
We likewise index power-conjugacy classes by the order of a representing element, always adding
letters as is traditional.  Thus $J = \{1A,2A,3AB\}$ for $A_4$ and  $J = \{1A,2A,3A,5AB\}$ for
$A_5$, with $3AB$ and $5AB$ each consisting of two conjugacy classes.

\subsection{The polytope $P_G$ of normalized nonnegative characters} 
\label{polytope-pg}
A general real-valued function $\phi \in \R(G^\sharp)$ has
an expansion $\sum x_i \chi_i$ with $x_i \in \R$.  The coefficients
are recovered via inner products, $x_i =
(\phi,\chi_i)/(\chi_i,\chi_i)$.  Alternative coordinates are given by
$y_j = \phi(C_j)$.  The  $\phi$ allowed for Theorem~\ref{thm2} are the
nonzero $\phi$ 
with the $x_i$ and $y_j$ non-negative integers.

An allowed $\phi$ gives the same lower bound  in Theorem~\ref{thm2} as
any of its positive multiples $m \phi$.  Without getting any new bounds, we can therefore give ourselves the convenience 
of allowing the $x_i$ and $y_j$ to be nonnegative rational numbers.    
Similarly, we can extend by continuity to allow the $x_i$ and $y_j$ to be nonnegative real numbers. 
The allowed $\phi$ then become the cone in $k$-dimensional Euclidean space given by 
$x_i \geq 0$ and $y_j \geq 0$, excluding the tip of the cone
at the origin.  

Writing the identity character as $\chi_1$, we can normalize via scaling to $x_1=1$.  Writing
the identity class as $C_{1A}$, the inequality $y_{1A} \geq 0$ is implied by
the other $y_j \geq 0$ and so the variable $y_{1A}$ can be ignored.  The polytope $P_G$ of
normalized nonnegative characters is then defined by $x_1=1$, the
inequalities $x_i \geq 0$ for $i \neq 1$, and inequalities $y_j \geq 0$ for $j \neq 1A$.
The point where all the  $x_i$ are zero 
is the unital character $\phi_1$.  The point where all the $y_j$ are zero
is the regular character $\phi_G$.    Thus the $(k-1)$-dimensional polytope
$P_G$ is determined by $2k-2$ linear inequalities,
with $k-1$ corresponding to non-unital characters and intersecting at $\phi_1$,
and $k-1$ corresponding to non-identity classes and intersecting at $\phi_G$.  

\begin{table}[htb]
\centering
\includegraphics{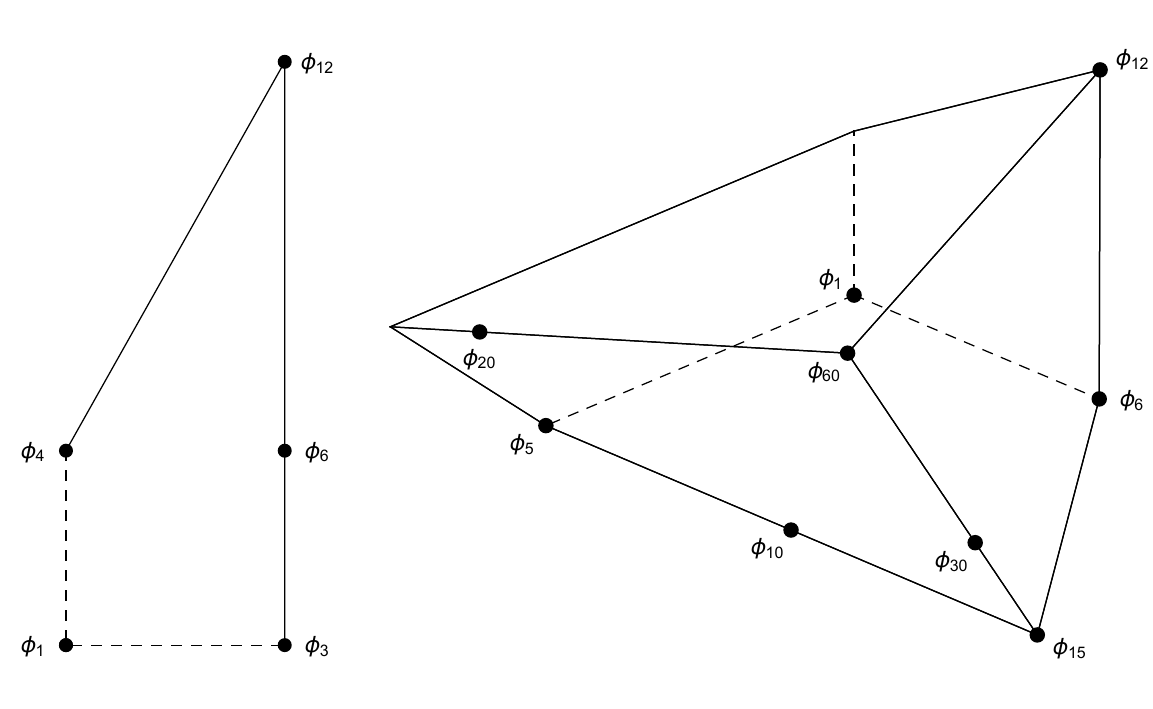}
\caption{\label{twopolytopes} The polytopes $P_{A_4}$ and $P_{A_5}$}
\end{table}

Figure~\ref{twopolytopes} continues our two examples.  On the left, $P_{A_4}$ is drawn in
the $x_2$-$x_3$ plane.  The character faces give the coordinate axes and are dashed.
The class faces are calculated from columns in the rational character table and are solid.  
   On the right, a view of $P_{A_5}$ is given in 
$x_{4}$-$x_{5}$-$x_{6}$ space.  The three pairwise intersections of character faces
give coordinate axes and 
are dashed, while all other edges are solid.   
In this view, the point $\phi_G = \phi_{60} = (4,5,3)$ 
should be considered as closest to the reader, with the solid lines visible
and the dashed lines hidden by the polytope.  Note that $P_{A_4}$ has the
combinatorics of a square and $P_{A_5}$ has the combinatorics of
a cube.    While the general $P_G$ is the intersection of an orthant with
tip $\phi_1$ and an orthant with tip $\phi_G$, its combinatorics are typically 
more complicated than $[0,1]^{(k-1)}$.  For example, the groups  $G=A_6$, $S_5$, $A_7$, and  $S_6$,
have $k=6$, $7$, $8$, and $11$ respectively; but instead of having $32$, $64$, $128$ and $1024$ vertices, their polytopes $P_G$ 
have $28$,
$40$, $115$, and $596$ vertices respectively.  

\subsection{Points in $P_G$}   
In the previous subsection, we have mentioned already the 
distinguished vertices $\phi_1$ and $\phi_G$.  For every rationally
irreducible character,
we also have $\phi_{\chi,L} = \chi + \widecheck{\chi}$, 
$\phi_{\chi,S} = \chi^2$, and $\phi_{\chi,Q} = \chi \chi^*$,
as in Section~\ref{choices}.    

For every subgroup $H$ of $G$, another element of $P_G$ is the permutation
character $\phi_{G/H}$.   For $H = G$, this character is just the $\phi_1$ considered
before, which is a vertex.  Otherwise, a theorem of Jordan, discussed at 
length in \cite{Ser03}, says that $\phi_{G/H}(C_j)=0$ 
for at least one $j$; in other words, $\phi_{G/H}$ is on at least one character face.  
For $A_4$ and $A_5$, there are respectively five and nine conjugacy classes of subgroups,
distinguished by their orders.  Figures~\ref{twopolytopes} draws the corresponding points, labeled by
 $\phi_{|G/H|}$.  All four vertices of $P_{A_4}$ and six of the eight vertices of $P_{A_5}$ are of the form $\phi_{N}$.   The remaining one $\phi_N$ in $P_{A_4}$ is on an edge, while the remaining three $\phi_N$ in $P_{A_5}$
 are on edges as well.

 \subsection{The best choice for $\phi$}  
Given $(G,c,\chi)$ and $u \in \{\walp,b,\alp\}$, 
let $m(G,c,\chi,u) = \max_{\phi \in P_G} m(G,c,\chi,\phi,u)$.  Computing these maxima seems difficult.  
Instead we vary $\phi$ over a modestly large finite set, denoting the 
largest bound appearing as $\mathfrak{d}(G,c,\chi,u)$. For most $G$, the set of $\phi$ we inspect consists of all $\phi_{\chi,L}$, $\phi_{\chi,S}$,
and $\phi_{\chi,Q}$, all $\phi_{G/H}$ including the regular character $\phi_G$, and
all vertices.  For some $G$, like $S_7$, there are too many vertices and we exclude them 
from the list of $\phi$ we try.  

For each $(G,\chi)$, we work either with $u=\walp$ or with $u=\alp$, as explained in 
the ``middle four columns" part of \S\ref{remainingrows}.  We then report $\mathfrak{d}(G,\chi) = \min_c \mathfrak{d}(G,c,\chi,u)$ 
in Section~\ref{Tables}.  
 
\section{The case $G=S_5$} 
\label{S5}
       Our focus in the next two 
sections is on finding initial segments
 $\cL(G,\chi; B)$ of complete lists of Artin $L$-functions, and
 in particular on finding the first root conductor $\delta_1(G,\chi)$.
It is a question of transferring  completeness 
statements for number fields to completeness statements for Artin 
$L$-functions via conductor relations.  In this section, we explain the process by
presenting the case $G=S_5$ in some detail.

\subsection{Different orders on the same set of fields}
     Consider the set $\cK$ of isomorphism classes of 
     quintic fields $K$ over $\Q$ with splitting field
 $L/\Q$ having Galois group $\gal(L/\Q) \cong S_5$.
 The group 
 $S_5$ has seven irreducible characters which we index by 
 degree and an auxiliary label: $\chi_{1a} = 1$, $\chi_{1b}$,
 $\chi_{4a}$, $\chi_{4b}$, $\chi_{5a}$, $\chi_{5b}$, and $\chi_{6a}$.  
  For $\phi$ a permutation character, let $D_\phi(K) = D(K_\phi)$ be the absolute
 discriminant of the associated resolvent algebra $K_\phi$ of $K$.   
 Extending by multiplicativity, functions $D_\chi : \cK \rightarrow \R_{>0}$ 
 are defined for general $\chi = \sum m_n \chi_n$.  They do not depend on the
 coefficient $m_{1a}$.    We follow our practice of often shifting attention to
 the corresponding root conductors $\delta_\chi(K) = D_\chi(K)^{1/\chi(e)}$.
 
\begin{table}[htb]
{\renewcommand{\arraycolsep}{3pt}
\[
\begin{array}{r|rrrrrrr|rrrrrrr|rr}
\lambda_5 & 1^5 & 2^2 1 & 31^2 & 5 & 21^3 & 41 & 32 &  1^5 & 2^2 1 & 31^2 & 5 & 21^3 & 41 & 32 & \\
\lambda_6 & 1^6 & 2^2 1^2 & 33 & 51 & 2^3 & 41^2 & 6 &  1^6 & \!\! 2^2 1^2 & 33 & 51 & 2^3 & 41^2 & 6 &
\walp(n) & \alp(n)  \\
\hline
\chi_{1a} &                    1 & 1 & 1 & 1 & 1 & 1 & 1 &    0 & 0 & 0 & 0 & 0 & 0 & 0 & \\
\chi_{1b} &                   1 & 1 & 1 & 1 & -1 & -1 & -1 &  0 & 0 & 0 & 0 & 1 & 1 & 1 & \\
\chi_{4a} &                    4 & 0 & 1 & -1 & {\bf 2} & 0 & -1 &  0 & 2 & 2 & 4 & {\bf 1} & 3 & 3 & 0.50 & 0.50  \\
\chi_{4b} &                    4 & 0 & {\bf 1} & -1 & -2 & 0 & {\bf 1} & 0 & 2 & {\bf 2} & 4 & 3 & 3 & 3 & 0.75 & 0.75  \\
\chi_{5a} &                    5 & {\bf 1} & -1 & 0 & {\bf 1} & -1 & {\bf 1} & 0 & {\bf 2} & 4 & 4 & {\bf 2} & 4 & 4 & 0.80 & 0.80 \\
\chi_{5b} &                    5 & {\bf 1} & -1 & 0 & -1 & {\bf 1} & -1 &         0 & {\bf 2} & 4 & 4 & 3 & 3 & 5 & 0.80 & 0.80  \\
\chi_{6a} &                    6 & -2 & 0 & {\bf 1} & 0 & 0 & 0 & 0 & 4 & 4 & {\bf 4} & 3 & 5 & 5 & 0.8\overline{3} & 0.8\overline{3} \\
\hline
\phi_{120} &                 120 & 0 & 0 & 0 & 0 & 0 & 0 & 0 &  60 &  80 & 96 & 60 & 90 & 100 & &  \\
 \end{array}
\]
}
\caption{\label{chartabs5} Standard character table of $S_5$ on the left, 
with entries $\chi_n(\tau)$;  tame table \cite[\S4.3]{jr-tame-wild},
on the right, with entries $c_\tau(\chi_n)$ as defined in \eqref{twosimilar}. } 
\end{table}  

Let $\cK(\chi; B) = \{K \in \cK : \delta_\chi(K) \leq B\}$.  Suppose now all 
 the $m_{n}$ are nonnegative with at least one coefficient
 besides $m_{1a}$  and $m_{1b}$ positive.  Then $\delta_\chi$ is a 
 height function in the sense that all the $\cK(\chi; B)$ 
 are finite.   Suppressing the secondary phenomenon 
 that ties among a finite number of fields can occur,
 we think of each $\delta_\chi$ as giving an ordering on 
 the set $\cK$.   
  
 The orderings coming from different $\delta_\chi$ can be very different.  For example, consider
 the field $K \in \cK$ defined by the polynomial $x^5 - 2x^4 + 4x^3 - 4x^2 + 2x - 4$.
 This field is the first field in $\cK$ when ordered by the regular character
 $\phi_{120} = \sum_n \chi_n(n) \chi_n$.   However it is 
 the $22^{\rm nd}$ field when
 ordered by $\phi_6  = 1 + \chi_{5b}$
 only the $2298^{\rm th}$ field when ordered by 
 $\phi_5 = 1+ \chi_{4a}$.
  
 This phenomenon of different orderings on the same set of number fields plays
 a prominent role in asymptotic studies \cite{wood}.  Here we are interested instead
 in initial segments and how they depend on $\chi$.  Our formalism lets
 us treat any $\chi$.  Following the conventions for general $G$ of the next section, 
 we focus on the five irreducible $\chi$ with
 $\chi(e)>1$, thus $\chi_n$ for $n \in \{4a,4b,5a,5b,6a\}$.   
 
\subsection{Computing Artin conductors} 
To compute general $D_\chi(K)$, one needs to work with enough resolvents of 
 $K=K_5 = \Q[x]/f_5(x)$.  For starters, we have the quadratic resolvent $K_2 = \Q[x]/(x^2-D(K_5))$ and the
 Cayley-Weber resolvent $K_6 = \Q[x]/f_6(x)$ \cite{generic,jr-tame-wild}.  The other resolvents 
 we will need are $K_{10} = K_5 \otimes K_2$, 
 $K_{12} = K_2 \otimes K_6$, and $K_{30} = K_5 \otimes K_6$.  Defining
 polynomials are obtained for $K_a \otimes K_b$ by the general formula 
\[
f_{ab}(x) = \prod_{i=1}^a \prod_{j=1}^b (x-\alpha_i-\beta_j),
\]
where $f_a(x)$ has roots $\alpha_i$ and $f_b(x)$ has roots $\beta_j$.   
So discriminants $D_2$, $D_5$, $D_6$, $D_{10}$, $D_{12}$, 
$D_{30}$ are easily computed.    
 
From the character table, the permutation characters $\phi_N$ in question are expressed
in the basis $\chi_n$ as on the left in the following display.  Inverting, one
gets the $\chi_n$ in terms of the $\phi_N$ as on the right.
\[
\begin{array}{r@{\:}l@{\:}l@{\qquad}r@{\:}l}
  \phi_2 &&  = 1 + \chi_{1b}, &\chi_{1b} & =  -1 + \phi_2,\\
  \phi_5 &&  = 1 + \chi_{4a},  & \chi_{4a} & =  -1 + \phi_{5}, \\
  \phi_6 && =  1 + \chi_{5b}, & \chi_{4a} & = 1 - \phi_2 - \phi_{5} + \phi_{10}, \\
 \phi_{10} & =  \phi_{5} \phi_2 & = 1 + \chi_{1b}  + \chi_{4a} + \chi_{4b}, & 
 \chi_{5a} & =   1 - \phi_2 - \phi_6 + \phi_{12},   \\
 \phi_{12} & =  \phi_6 \phi_2 & = 1 + \chi_{1b} + \chi_{5a} + \chi_{5b}, & \chi_{5b} & =  -1 + \phi_6,   \\
 \phi_{30} &  = \phi_5 \phi_6 & = 1 + 2 \chi_{4a} + 2 \chi_{5a} + \chi_{5b} + \chi_{6a}, \!\!\!\!  & \chi_{6a}   & =   2 \phi_2 - 2 \phi_5 + \phi_6 - 2 \phi_{12} + \phi_{30}. 
\end{array}
\]
Conductors $D_n$ belonging to the $\chi_n$ are calculable through these formulas, as e.g.\
$D_{6a}   =  D_2^2 D_5^{-2} D_6 D_{12}^{-2} D_{30}$.    

For all the groups $G$ considered in the next section, we proceeded
similarly.  Thus we started with rational character tables from {\em
  Magma}.  We used linear algebra to express rationally irreducible
characters in terms of permutation characters.  We used {\em Magma}
again to compute resolvents and then {\em Pari} to evaluate their
discriminants.  In this last step, we often confronted large degree
polynomials with large coefficients.  The discriminant computation was
only feasible because we knew {\em a priori} the set of primes
dividing the discriminant, and could then easily compute the $p$-parts
of the discriminants of these resolvent fields for relevant primes $p$
using {\em Pari/gp} without fully factoring the discriminants of the
resolvent polynomials.

{\em Magma}'s Artin representation package computes conductors of
Artin representations in a different and more local manner.
Presently, it does not compute all conductors in our range because
some decomposition groups are too large.

\subsection{Transferring completeness results.}
As an initial complete list
of fields, we take $\cK(\phi ; 85)$ with $\phi=\phi_G=\phi_{120}$.
We know from \cite{jr-global-database}
that this set consists of $2080$ fields.
We list these fields by increasing discriminant,  $K^1$, \dots, $K^{2080}$, with
the resolution of ties conveniently not affecting the explicit results appearing in 
Table~\ref{tablelabel1}. 

The quantities of Section~\ref{Type} reappear here, and we will use the abbreviations 
$\walp(n) = \walp(S_5,\chi_n,\phi)$ and 
$\alp(n) =  \alp(S_5,\chi_n,\phi)$.
Since $\phi$ is zero outside of the identity class, the formulas
simplify substantially:
\begin{align*}
\walp(n) & =\frac{\phi(e)}{\chi_n(e)}  \min_\tau \frac{\chi_n(e)-\chi_n(\tau)}{\phi(e)-\phi(\tau)} =  1-\max_{\tau}  \frac{\chi_n(\tau)}{n},  \\
\alp(n)& =   \frac{\phi(e)}{\chi_n(e)}   \min_{\tau} \frac{c_\tau(\chi_{n})}{c_\tau(\phi)}  = \frac{1}{n} \min_{\tau}
 \frac{c_\tau(\chi_{n}) \overline{\tau}}{\overline{\tau} - 1}.
\end{align*}
For each of the five $n$, the classes contributing to the minima are
in bold on Table~\ref{chartabs5}.
So, extremely simply, for computing $\walp(n)$ on the left, the largest $\chi_n(\tau)$ besides $\chi_n(e)$ are in bold.  
For computing $\alp(n)$ on the right, the $c_\tau(\chi_n)$  with 
$c_\tau(\chi_n)/c_\tau(\phi)$ minimized
are put in bold.   For the group $S_5$, one has agreement 
$\walp(n) = \alp(n)$ in all five cases.
This equality occurs for 170 of the lines in Tables~\ref{tablelabel1}--\ref{tablelabel8}, with the other possibility $\walp(n) < \alp(n)$
occurring for the remaining 25 lines.  

For any cutoff $B$, conductor relations give
\[ 
\cK(\chi_n;B^{\walp(n)}) \subseteq \cK(\phi; B).
\]
One has an analogous inclusion for general $(G,\chi)$, with $\phi$ again the
regular character for $G$.   When $G$ satisfies the tame-wild principle 
of \cite{jr-tame-wild}, the $\walp$ in exponents can be replaced by $\alp$.
The group $S_5$ does satisfy the tame-wild principle, but in this case
 the replacement has no effect.  

The final results are on Table~\ref{tablelabel1}.  In particular for $n = 4a$, $4b$, $5a$, $5b$, $6a$ the 
unique minimizing fields are $K^{103}$, $K^{21}$, $K^{14}$, $K^{6}$, and $K^{12}$,
with root conductors approximately $6.33$, $18.72$, $17.78$, $16.27$, and $18.18$.   
The lengths of the initial segments identified are $45$, $15$, $186$, $592$, and $110$.
Note that because of the relations 
$\phi_5 = 1+ \chi_{4a}$ and $\phi_6  = 1 + \chi_{5b}$, the results for $4a$ and $5b$ 
are just translation of known minima of discriminants of number fields
with Galois groups $5T5$ and $6T14$ respectively.  For $4b$, $5b$,
$6a$, and the majority
of the characters on the tables of the next section, the first root conductor
and the entire initial segment are new.

\section{Tables for $84$ groups $G$}
\label{Tables}
In this section, we present our computational results for small
Galois types.   For 
simplicity, we focus on results coming from complete lists of Galois
number fields.  Summarizing statements are given in
\S\ref{twotheorems} and then many more details in
\S\ref{table-org}.

\subsection{Lower bounds and initial segments}
\label{twotheorems}
We consider all groups with a faithful transitive permutation 
representation in some degree from two to nine, except we exclude the 
nonsolvable groups in degrees eight and nine.   There are $84$ such groups, 
and we consider all associated 
Galois types $(G,\chi)$ with $\chi$ a rationally irreducible faithful character.
Our first result gives conditional lower bounds: 

\begin{thm} \label{bounds-are-right}
  For each of the $195$ Galois types  $(G,\chi)$ listed in
  Tables~\ref{tablelabel1}--\ref{tablelabel8}, the listed value 
  $\mathfrak{d}$ gives a lower bound for the root conductor of all
  Artin representations of type $(G,\chi)$,
  assuming the Artin conjecture and Riemann hypothesis for relevant
  $L$-functions.
\end{thm}

The bounds in Tables~\ref{tablelabel1}--\ref{tablelabel8} are graphed 
with the best previously known 
\begin{figure}[htb]
\centering
\includegraphics[width=4.5in]{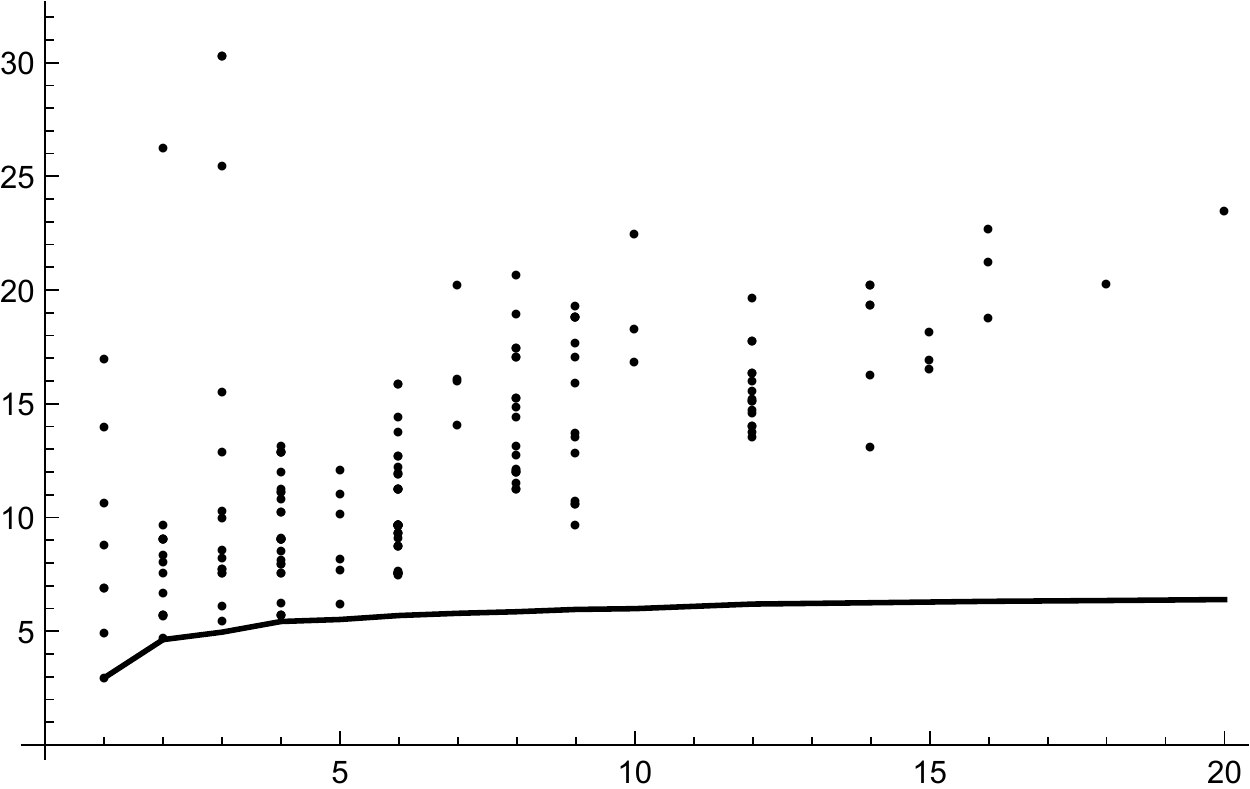}
\caption{Points $(\chi_1(e),\mathfrak{d}(G,\chi))$ present lower bounds
from Tables~\ref{tablelabel1}--\ref{tablelabel8}. 
The piecewise-linear curve plots lower bounds from \cite{PM}.  Both the points 
and the curve
assume the
Artin conjecture and Riemann hypothesis for the relevant $L$-functions.
\label{amalia-plot}}
\end{figure}
bounds from \cite{PM} in Figure~\ref{amalia-plot}.  
The horizontal
axis represents the dimension $n_1=\chi_1(e)$ of any irreducible constituent 
$\chi_1$ of $\chi$.  The vertical axis corresponds to lower bounds on
root conductors.  The
piecewise-linear curve connects  bounds from \cite{PM}, and
there is one dot at height $\mathfrak{d}(G,\chi)$ for each $(G,\chi)$ from
Tables~\ref{tablelabel1}--\ref{tablelabel8} with $\chi_1(e)\leq 20$.
Here we are freely passing back and forth between a rational character
$\chi$ and
an irreducible constituent $\chi_1$ via $\delta_1(G,\chi) =
\delta_1(G,\chi_1)$, which is a direct consequence of \eqref{discequal}.

Not surprisingly, the type-based bounds are larger.  In low 
dimensions $n_1$, some type-based bounds are close to the general bounds,
but by dimension $5$ there is a clear separation which widens as the
dimension grows.  This may in part be explained by the fact that we
are only seeing a small number of representations for each of these
dimensions.  However, as we explain in \S\ref{speculation},  
we also expect that the asymptotic lower bound
 of $\sqrt{\Omega} \approx 6.7$ \cite{PM} is not optimal, and that 
this bound is more likely to be at least
$\Omega\approx 44.8$.

Our second result unconditionally identifies initial segments:
\begin{thm} \label{minima-and-segments-are-right} For $144$ Galois
  types $(G,\chi)$, Tables~\ref{tablelabel1}--\ref{tablelabel8} identify
  a non-empty initial segment  $\cL(G,\chi;B^\beta)$, and in particular
  identify the minimal root conductor $\delta_1(G,\chi)$.
\end{thm}

\subsection{Tables detailing results on lower bounds and initial segments}    
\label{table-org} 
Our tables are organized by the standard doubly-indexed lists of 
transitive permutation groups $mTj$, with degrees $m$ running from $2$
through $9$.  Within a degree, the blocks of rows are indexed by
increasing $j$.  There is no block to print if $mTj$ has no faithful
irreducible characters.  For example, there is no block to print for
groups having noncyclic center, such as $4T2 = V = C_2 \times C_2$ or
$8T9 = D_4 \times C_2$.  Also the block belonging to $mTj$ is omitted
if the abstract group $G$ underlying $mTj$ has appeared earlier.  For
example $G=S_4$ has four transitive realization in degrees $m \leq 8$,
namely $4T5$, $6T7$, $6T8$, and $8T14$; there is correspondingly a
$4T5$ line on our tables, but no $6T7$, $6T8$, or $8T14$ lines.  

\input{foo}

\subsubsection{Top row of the $G$-block.}
The top row in the $G$-block is different from the other rows, as it
gives information corresponding to  
the abstract group $G$.  Instead of referring to a faithful
irreducible character, as the  
other lines do, many of its entries are the corresponding quantities
for the regular character 
$\phi_G$.  The first four entries are a common name for the group
$G$ (if there is one), the order 
$\phi_G(e) = |G|$,  the symbol TW if $G$ is known to have the
universal tame-wild property as defined in 
\cite{jr-tame-wild},
and finally $k,N$.
Here, $k$ is the size of the rational 
character table, and $N$ is number of vertices of the polytope $P_G$ discussed
in \S\ref{polytope-pg}, or a dash if we did not compute $N$.  
The last four entries are the
smallest root discriminant of a Galois $G$ field, the factored form of
the corresponding discriminant, 
a cutoff $B$ for which the set $\cK(G;B)$ is known, and the 
size $|\cK(G;B)|$.  
  
  \subsubsection{Remaining rows of the $G$-block}    
  \label{remainingrows}
  Each remaining line of the $G$-block   
 corresponds to a type $(G,\chi)$.  
 However the number of rows in the $G$-block is typically substantially less than the 
number of faithful irreducible characters of $G$, as we list only one
representative of each $\Gal(\overline{\Q}/\Q) \times \mbox{Out}(G)$ orbit 
of such characters.   As an example, $S_6$ has eleven characters, all rational.  
Of the nine which are faithful, there are three which are fixed by the 
nontrivial element of $\mbox{Out}(S_6)$ and the others form three
two-element orbits.  Thus the $S_6$-block has six rows.   
In general, the information on a $(G,\chi)$ row comes in three parts,
which we now describe in turn.

 {\em First four columns.}   
The first column gives the lexicographically first permutation group
$mTj$ for which the corresponding permutation character has $\chi$
as a rational constituent.  
Then $n_1=\chi_1(e)$ is the degree of an absolutely irreducible
character $\chi_1$ such that $\chi$ is the 
sum of its conjugates.  The number $n_1$ is superscripted by the size
of the $\Out(G)$ orbit of 
$\chi$, in the unusual case when this orbit size is not $1$.  
Next, the complex number $z$ is a generator for the field generated by
the values of the character $\chi_1$, 
with no number printed in the common case that $\chi_1$ is rational-valued.  
The last entry gives the interval $\Imnaught$, where
$\widecheck{\chi}$ and $\widehat{\chi}$
are the numbers  introduced in the beginning of Section~\ref{choices}.  
 In the range presented, the data of $mTj$, $n_1$, $z$, and $\Imnaught$
 suffice to distinguish Galois types $(G,\chi)$ from each other.

{\em Middle four columns.}   The next four columns focus on 
 minimal root conductors.  In the first entry, $\frak d$ is the best 
 conditional lower bound we obtained for root conductors,
and the subscript $i \in \{\ell,s,q,g,p,v\}$ gives information on
   the corresponding auxiliary character $\phi$.  The first four possibilities refer to the methods of
   Section~\ref{choices}, namely {\em linear}, {\em square}, {\em quadratic},
   and {\em Galois}.  The last two, $p$ and $v$, indicate a {\em permutation}
   character and a character coming from a {\em vertex} of the polytope $P_G$.  
   The best $\phi$ of the ones  
   we inspect is always at a vertex, except in the three cases on
   Table~\ref{tablelabel2} 
   where $*$ is appended to the subscript.  
  Capital letters
   $S$, $Q$, $G$, $P$, and $V$ also appear as subscripts.  
   These occur only for groups marked with TW, and indicate
   that the tame-wild principle improved the lower bound.   
   For most groups with fifteen or more classes, it was prohibitive
   to calculate all vertices, and the best of the other methods
   is indicated.

When the second entry is in roman type, it is 
the minimal root conductor and
the third entry is the minimal conductor in factored form.
When the second entry is in italic type, then it is the smallest
currently known root conductor.
The fourth entry gives the position of the source number field on the 
complete list ordered by Galois root discriminant.  This information
lets readers obtain further information from \cite{jr-global-database},
such as a defining polynomial and details on ramification.

  {\em Last three columns.}  The quantity $\beta$ is 
   the exponent we are using to pass from Galois number fields to Artin representations.
   Writing $\walp = \walp(G,\chi,\phi_G)$
   and $\alp = \alp(G,\chi,\phi_G)$,
   one has the universal relation $\walp \leq \alp$. 
   When equality holds then the common number is printed.   
   To indicate that inequality holds, an extra symbol is printed.
   When we know that $G$ satisfies TW then we can use larger
   exponent and $\alp_\twgood$ is printed.  
   Otherwise we use the smaller exponent and $\walp_\twbad$ is printed.  
    The column
   $B^\beta$ gives the corresponding upper bound on our
   complete list of root conductors.  Finally 
   the column $\#$ gives $|\cL(G,\chi; B^\beta)|$,
   the length of the complete list of Artin $L$-functions
   we have identified.   For the $L$-functions themselves,
   we refer to \cite{LMFDB}.
   
   \section{Discussion of tables}
   \label{discussion}
   In this section, we discuss four topics, each of which makes 
   specific reference to parts of the tables of the 
   previous section.  Each of the topics also serves the
   general purpose of making the 
   tables more readily understandable.   
   
\subsection{Comparison of first Galois root discriminants and root conductors}  
   Suppose first, for notational simplicity, that $G$ is a group for which
 all irreducible complex characters take rational values only.   
 When one fixes $K^{\rm gal}$ with $\Gal(K^{\rm gal}/\Q) \cong G$
  and lets $\chi$ runs over all the irreducible characters of 
$G$, the root discriminant $\delta_{\rm Gal}$ is just the 
weighted multiplicative average $\prod_\chi = {\delta_\chi^{\chi(e)^2/|G|}}$. 
 Deviation of a root conductor $\delta_\chi$ from $\delta_{\rm Gal}$ is
 caused by nonzero values of $\chi$.    
 When $\chi(e)$ is large and
$\Imnaught$ is small,  $\delta_\chi$ 
 is necessarily close to $\delta_{\rm Gal}$.   One
 can therefore think  generally of $\delta_{\rm Gal}$ as a first 
 approximation to $\delta_{\rm \chi}$.   The general principle
 of $\delta_{\rm Gal}$ approximating $\delta_{\rm \chi}$ 
 applies to groups $G$ with irrational characters as well.     
  
 Our first example of $S_5$ illustrates both how the principle
  $\delta_{\rm Gal} \approx \delta_{\chi}$ is reflected in the tables, 
  and how it tends to be somewhat off in the direction that 
  $\delta_{\rm Gal} > \delta_{\chi}$.   
 For a given $K^{\rm gal}$, the variance of its $\delta_\chi$ about its
 $\delta_{\rm Gal}$ is substantial and depends on the details
 of the ramification in $K^{\rm gal}$.    There are many $K^{\rm gal}$ 
 with root discriminant near the minimal root discriminant, all of which
 are possible sources of minimal root conductors.   It
 is therefore expected that the minimal conductors $\delta_1(S_5,\chi) = \min \delta_\chi$
 printed in the table, $6.33$, $18.72$, $16.27$, $17.78$, and $18.18$,
   are substantially less than the printed 
 minimal root discriminant $\delta_1(S_5,\phi_{120}) \approx 24.18$.   As groups $G$
 get larger, one can generally expect tighter clustering 
 of the $\delta_1(G,\chi)$ about $\delta_1(G,\phi_G)$.   One can see
 the beginning of this trend in our partial results for
 $S_6$ and $S_7$.

  \subsection{Known and unknown minimal root conductors}  Our method of starting with a complete list of 
 Galois fields is motivated by the principle from the previous
 subsection that the Galois root discriminant $\delta_{\rm Gal}$ is 
 a natural first approximation to $\delta_\chi$.   Indeed, as the tables show via nonzero entries in the $\#$
 column, this general method suffices to obtain a non-empty initial segment for most $(G,\chi)$.  
 As our focus is primarily on the first root conductor $\delta_1 = \delta_1(G,\chi)$,  we do not pursue larger
 initial segments in these cases.
  
 When the initial segment from our general method is empty, as reported by a $0$ in the $\#$ 
 column, we aim to nonetheless present the minimal root conductor $\delta_1$.  
 Suppose there are subgroups $H_m \subset H_k  \subseteq G$, of the 
 indicated indices, such that a multiple of the character $\chi$ of interest
 is a difference of the corresponding permutation characters: 
$c \chi = \phi_m - \phi_k$.
 Suppose one has the complete list of all degree $m$ fields
 corresponding to the permutation representation of $G$ on $G/H_m$ and 
 root discriminant $\leq B$.  Then one can extract the complete
 list of $\cL(G,\chi;B^{m/(m-k)})$ of desired Artin $L$-functions. 

For example, consider $\chi_5$, the absolutely irreducible $5$-dimensional character of
$A_6$.  The permutation character for a corresponding sextic field
decomposes $\phi_6 = 1+\chi_5$, and so the discriminant of the sextic
field equals the conductor of $\chi_5$.  As an example with $k>1$, 
 consider the  $6$-dimension character $\chi_6$ for
$C_3\wr C_3 = 9T17$, which is the sum of a three-dimensional character
and its conjugate.    The nonic field has a cubic subfield, and the
characters are related by $\phi_9 = \phi_3 + \chi_6$.  In terms of conductors, $D_9 = D_3 \cdot
D_{\chi_6}$, where $D_9$ and $D_3$ are field discriminants.  So, we
can determine the minimal conductor of an $L$-function with type $(C_3\wr C_3,\chi_6)$  from 
a sufficiently long complete list of nonic number fields with Galois group $C_3\wr C_3$.

 This method, applied to both old and newer lists presented in \cite{jr-global-database},
 accounts for all but one of the $\delta_1$ reported
 in Roman type on the same line as a $0$ in the $\#$ column.  
  The remaining case of an established $\delta_1$ is for the type $(\GL_3(2),\chi_7)$.  The
 group $\GL_3(2)$ appears on our tables as $7T5$.   The permutation 
 representation $8T37$ has character $\chi_7+1$.     Here
 the general method says that $\cL(\GL_3(2),\chi_7; 26.12)$ is empty.  
 It is prohibitive to compute the first octic discriminant by 
 searching among octic polynomials.   In \cite{jr-psl27} we carried out a long search of septic polynomials,
 examining all local possibilities giving an octic discriminant at most $30$.   
 This computation shows that $|\cL(\GL_3(2),\chi_7; 48.76)| = 25$ and in particular identifies
 $\delta_1 = 21^{8/7} \approx 32.44$.   
 
      The complete lists of Galois fields for a group first appearing in degree $m$ were likewise computed
 by searching polynomials in degree $m$, targeting for small $\delta_{\rm Gal}$.  This single
 search can give many first root conductors at once.  For example, the largest groups
 on our octic and nonic tables are $S_4 \wr S_2 = 8T47$  and $S_3 \wr S_3 = 9T31$.  
 In these cases, minimal root conductors were obtained for $5$ of the $10$ and 
 $7$ of the $12$ faithful $\chi$ respectively.   Searches adapted to a particular character 
 $\chi$ as in the previous paragraph 
 can be viewed as a refinement of our method, with targeting being 
 not for small $\delta_{\rm Gal}$ but instead for small $\delta_\chi$.  
 Many of the italicized entries in the column $\delta_1$ seem 
 improvable to known minimal root conductors via this refinement.

\subsection{The ratio $\delta_1/{\frak d}$}
In all cases on the table, $\delta_1>{\frak d}$. 
Thus, as expected, we did not encounter a contradiction to 
the Artin conjecture or the Riemann hypothesis. 
In some cases on the table, the ratio $\delta_1/{\frak d}$ is quite
close to $1$.
As two series of examples, consider $S_m$ with its reflection character
$\chi_{m-1} = \phi_m-1$, and $D_m$ and the sum $\chi$ of all its 
faithful $2$-dimensional characters.
Then these ratios are as  follows:
 \[
 \begin{array}{r|llllllll}
 m & \;\; 2 & \;\; 3 & \; \;4 & \; \; 5 & \; \;  6 & \; \; 7 & \; \;8 & \; \;9  \\
 \hline
 \delta_1/{\frak d}\mbox{ for } (S_m,\chi_{m-1}) & 1.00 & 1.02 & 1.2 & 1.005 & 1.1 & 1.007 \\
\delta_1/{\frak d} \mbox{ for } (D_m,\chi) &      &    &  1.09 &  1.02 & 1.23 & 1.006 &  1.07 & 1.45 \\
\end{array}
 \]
 In the cases with the smallest ratios, 
 the transition
 from no $L$-functions to many $L$-functions is commonly abrupt.  For example, 
 in the case $(S_7,\chi_6)$ the lower bound is ${\frak d} \approx 7.50$ and the first seven
 rounded root conductors are $7.55$, $7.60$, $7.61$, $7.62$, $7.64$, $7.66$, and $7.66$.
 
 When the translation from no $L$-functions to many $L$-functions is not 
 abrupt, but there is an $L$-function with outlyingly small conductor, again
 $\delta_1/{\frak d}$ may be quite close to $1$.  As an example,  
 for $(8T25,\chi_7)$, one has ${\frak d} \approx 16.10$ and $\delta_1 = 29^{6/7} \approx 17.93$
 yielding $\delta_1/{\frak d} \approx 1.11$.    However in this
 case the next root conductor is $\delta_2 = 113^{6/7} \approx 57.52$, yielding 
 $\delta_2/{\frak d} \approx 3.57$.  Thus the close agreement is entirely dependent
 on the $L$-function with outlyingly small conductor.   Even the second
 root conductor is somewhat of an outlier as the next three conductors
 are  $71.70$, $76.39$, and $76.39$, so that already $\delta_3/{\frak d} \approx  4.45$.  
 
 There are many $(G,\chi)$ on the table for which the ratio $\delta_1/{\frak d}$ is
 around $2$ or $3$. 
 There is some room for improvement in our analytic lower bounds,  
for example changing the test function \eqref{Odlyzko}, varying $\phi$ over
all of $P_G$, or replacing the exponent $\walp$ with
the best possible exponent $b$.   However examples like the one
in the previous paragraph suggest to us that in many cases the
resulting increase in $\frak d$ towards $\delta_1$ would be very small.

 \subsection{Multiply minimal fields}   
Tables~\ref{tablelabel1}--\ref{tablelabel8} make implicit reference to many
 Galois number fields, and all necessary complete lists are accessible on 
 the database \cite{jr-global-database}.   Table~\ref{multmin} presents
 a small excerpt from this database  by giving six polynomials $f(x)$.  For each $f(x)$,
 Table~\ref{multmin} first gives the Galois group $G$ and the root discriminant $\delta$ of the 
 splitting field $K^{\rm gal}$.   We are highlighting these particular 
Galois number fields $K^{\rm gal}$ here
 because  they are {\em multiply minimal}: they each give rise to the minimal root conductor for at least 
 two different rationally irreducible characters $\chi$.    The 
 degrees of these characters are given in the last column of Table~\ref{multmin}.  
  \begin{table}[htb]
 \[
 {\renewcommand{\arraycolsep}{4pt}
 \begin{array}{crclll}
 G &  & \!\! \delta \!\!  & & \mbox{Polynomial} & \chi(e)  \\
 \hline
 A_5 & 2^{3/2} 17^{2/3} & \!\! \approx \!\! & 18.70 &  x^5-x^4+2 x^2-2 x+2  &4,6 \\
 A_6 & 2^{13/6} 3^{16/9}                         & \!\! \approx \!\! &   31.66               & x^6-3 x^4-12 x^3-9 x^2+1 &9,10,16 \\
 S_6 & 11^{1/2} 41^{2/3} & \approx & 39.44 & x^6 - x^5 - 4x^4 + 6 x^3 - 6x + 5 & 5,10 \\
 \SL_2(3) & 163^{2/3} &  \!\! \approx \!\!  & 29.83 &  x^8+9 x^6+23 x^4+14 x^2+1 & 2,4  \\ 
 8T47 & 2^{31/24} 5^{1/2} 41^{1/2} &  \!\! \approx \!\! & 35.05 & x^8-2 x^7+6 x^6-2 x^5+26 x^4&12,18\\
         &   & &     & \; -24 x^3-24 x^2+16 x+4\\
 9T19 & 3^{37/24} 7^{3/4} & \!\! \approx \!\!  & 23.41 & x^9-3 x^8-3 x^7+12 x^6-21 x^5 &8,8\\
           &                                          &           &            & \;+ 36 x^4-48 x^3+45 x^2-24 x+7 \\
\end{array} 
}
 \]
 \caption{\label{multmin}  Invariants and defining polynomials for Galois
 number fields giving rise to minimal root discriminants for 
 at least two rationally irreducible characters $\chi$ }
 \end{table}
 
 Further information on the characters $\chi$ is given in Tables \ref{tablelabel1}--\ref{tablelabel8}.
 An interesting point, evident from repeated $1$'s in the $G$-block on these tables, 
 is that five of the six fields $K^{\rm gal}$ are also first on the list of $G$ fields 
 ordered by root discriminant.  The exception is 
 the $S_6$ field on Table~\ref{multmin}, which is only sixth on the
 list of Galois $S_6$ fields ordered by root discriminant.

 \section{Lower bounds in large degrees}  
\label{asymp} 
In this section, we 
continue our practice of assuming the Artin conjecture and Riemann hypothesis
for the relevant $L$-functions.  For $n$ a positive integer, let
$\Delta_1(n)$ be the smallest root discriminant of a degree $n$ field.   As illustrated by Figure~\ref{contourM}, one has, 
\begin{equation} 
\label{lowfield}
\liminf_{n \to \infty} \Delta_1(n) \geq \Omega \approx 44.7632.
\end{equation}
Now let $\delta_1(n)$ be the smallest root conductor
of an absolutely irreducible degree $n$ Artin representation.  Theorem~4.2 of \cite{PM} uses the quadratic method to conclude that 
 $\delta_1(n) \geq 6.59 e^{(-13278.42/n)^2}$.  If one repeats the argument there without concerns for effectivity, one gets
 \begin{equation}
 \label{lowrep}
\liminf_{n \to \infty} \delta_1(n) \geq \sqrt{\Omega} \approx 6.6905.
\end{equation}
The contrast between \eqref{lowfield} and \eqref{lowrep} is striking,  and raises
the question of whether $\sqrt{\Omega}$ in \eqref{lowrep} can be increased at least part way 
to $\Omega$.    

\subsection{The constant $\Omega$ as a limiting lower bound.} 
The next corollary makes use of the extreme character values
$\widecheck{\chi}$ and $\widehat{\chi}$ introduced at the beginning of
Section~\ref{choices}.  It shows that if one restricts the type, then
one can indeed increase $\sqrt{\Omega}$ all the way to $\Omega$.  We
formulate the corollary in the context of rationally irreducible
characters, to stay in the main context we have set up.  However via
\eqref{discequal}, it can be translated to a statement about
absolutely irreducible characters.

\begin{cor} 
  \label{limitcor} Let $(G_k,\chi_k)$ be a sequence of rationally
  irreducible Galois types of degree $n_k = \chi_k(e)$ .  Suppose that
  the number of irreducible constituents $(\chi_k,\chi_k)$ is bounded,
  $n_k \rightarrow \infty$, and either
\begin{description}
\item[A] $\widecheck{\chi}_k/n_k \rightarrow 0$, or
\item[B] $\widehat{\chi}_k/n_k \rightarrow 0$.
\end{description}  
Then, assuming the Artin conjecture and Riemann hypothesis for
relevant $L$-functions,
\begin{equation}
\label{limlowerbound}
\liminf_{k \to \infty} \delta_1(G_k,\chi_k) \geq \Omega.
\end{equation}
\end{cor}
\begin{proof}
For Case A, Theorem~\ref{thm2} using a linear auxiliary character as in \eqref{formlinear} says 
\[
\delta_1(G_k,\chi_k)  \geq  M \left( \frac{n_k}{\widecheck{\chi}_k} + 1, \frac{r_k}{\widecheck{\chi}_k} + 1,(\chi_k,\chi_k)  \right)^{{1+\widecheck{\chi}_k/{n_k}} }.
\]
For Case B, Theorem~\ref{thm2} using a Galois auxiliary character as in \eqref{formula3a} says
\[
\delta_1(G_k,\chi_k)  \geq  M \left( |G_k| , 0 ,(\chi_k,\chi_k) \right)^{1 - {\widehat{\chi}}/{n_k}}.
\]
In both cases, the first argument of $M$ tends to infinity, the second
argument does not matter, the third argument does not matter either by
boundedness, and the exponent tends to $1$.  By \eqref{asymptotic1},
these right sides thus have an infimum limit of at least $\Omega$,
giving the conclusion \eqref{limlowerbound}.
\end{proof}

For the proof of Case~B, the square auxiliary character would work
equally well through \eqref{formsquare}.  Also~\eqref{lowfield},
\eqref{lowrep}, and Corollary~\ref{limitcor} could all be strengthened
by considering the placement of complex conjugation.  For example,
when restricting to the totally real case $c=e$, the $\Omega$'s in
\eqref{lowfield}, \eqref{lowrep}, and \eqref{limlowerbound} are simply
replaced by $\Theta \approx 215.3325$.

\subsection{Four contrasting examples.} Many natural sequences of types are covered by either Hypothesis A or Hypothesis B, but some are not.
 Table~\ref{limitcor} summarizes four sequences which we discuss together with some related sequences next.

\begin{table}[htb]
{
\[
\begin{array}{c| l |cccc|cc}
G_k & \multicolumn{1}{c|}{\chi_k} & \widecheck{\chi}_k  & \widehat{\chi}_k & n & |G| & \mbox{A}& \mbox{B}  \\
\hline 
   \PGL_2(k) &\mbox{Steinberg}        &         1 & 1 & k & k^3-k & \checkmark & \checkmark \\
          S_{k} &   \mbox{Reflection}     &          1 & k-3 & k-1 & k! & \checkmark \\
    2_\epsilon^{1+2k} &  \mbox{Spin} &               2^k  & 0 & 2^k & 2^{1+2k} & & \checkmark \\
                  2^k.S_k &  \mbox{Reflection} &               k &  k-2 & k & 2^k k! &   &   
  \end{array}
\]
}
\caption{\label{fourasymp} Four sequences of types, with Corollary~\ref{limitcor} applicable to the first three.}
\end{table}

\subsubsection{The group $\PGL_2(k)$ and its characters of degree $k-1$, $k$, and $k+1$.}  
\label{pglk}
In the sequence $(\PGL_2(k),\chi_k)$ from the first line of Table~\ref{fourasymp}, the index 
$k$ is restricted to be a prime power.  The
permutation character $\phi_{k+1}$ arising from the natural action
of $\PGL_2(k)$ on $\mathbb{P}^1(\mathbb{F}_k)$ decomposes as
$1+\chi_k$ where $\chi_k$ is the Steinberg character.
Table~\ref{fourasymp} says that the ratios $\widecheck{\chi}_k/n_k$ and  $\widehat{\chi}_k/n_k$ are both $1/k$,
so Corollary~\ref{limitcor}
applies through both Hypotheses A and B.  

The conductor of $\chi_k$ is the absolute discriminant of
the degree $k+1$ number field with character $\phi_{k+1}$.  Thus, in this instance, \eqref{limlowerbound} is
already implied by the classical \eqref{lowfield}.  However, the other nonabelian 
irreducible characters $\chi$ of $\PGL_2(k)$ behave very similarly 
to $\chi_k$.  Their dimensions are in $\{k-1,k,k+1\}$  and 
their values besides $\chi(e)$ are all in $[-2,2]$.  
So suppose for each $k$, an arbitrary nonabelian rationally
irreducible character $\chi_k$ of $\PGL_2(k)$ were chosen, in such a
way that the sequence $(\chi_k,\chi_k)$ is bounded.  Then
Corollary~\ref{limitcor} would again apply through both Hypotheses A
and B.  But now the $\chi_k$ are not particularly closely related to
permutation characters.

\subsubsection{The group $S_{k}$ and its canonical characters.} 
\label{sk}    
As with the last 
example, the permutation character $\phi_{k}$ arising from the natural action
of $S_{k}$ on $\{1,\dots,k\}$ decomposes as $1 + \chi_k$ where $\chi_k$ 
is the reflection character with degree $k-1$.  The second line of Table~\ref{fourasymp} shows that 
Corollary~\ref{limitcor} applies through Hypothesis A.  In fact, using the linear
auxiliary character underlying Hypothesis A here is essential; 
the limiting lower bound coming from the square or quadratic auxiliary characters is 
$\sqrt{\Omega}$, and this lower bound is just $1$ from the Galois auxiliary character.

Again in parallel to the 
previous example, the familiar sequence $(S_k,\chi_k)$ of types needs
to be modified to make it a good illustration of the applicability of Corollary~\ref{limitcor}.
 Characters of $S_k$ are most commonly indexed by partitions of $k$,
with $\chi_{(k)} =1$, $\chi_{(k-1,1)}$ being the reflection character, and $\chi_{(1,1,\dots,1,1)}$ 
being the sign character.  However an alternative convention is to include explicit reference to the 
degree $k$ and then omit the largest part of the partition, so that the above three
characters have the alternative names $\chi_{k,()}$, $\chi_{k,(1)}$, and $\chi_{k,(1,\dots,1,1)}$.  
With this convention, one can prove that for any fixed partition $\mu$ of a positive integer $m$, the sequence 
of types $(G_k,\chi_{k,\mu})$ satisfies Hypothesis A but not B.   

The case of general $\mu$ is well represented by the two cases where $m=2$.  In these two cases, information
in the same format as Table~\eqref{fourasymp} is
\[
{\def\arraystretch{1.3}
\begin{array}{c| l |cll}
G_k & \multicolumn{1}{c|}{\chi_k} & \multicolumn{1}{c}{\widecheck{\chi}_k}  &  \multicolumn{1}{c}{\widehat{\chi}_k} &  \multicolumn{1}{c}{n}  \\
\hline 
          S_{k} &   \chi_{k,(1,1)}    &   {\lfloor \frac{k-1}{2} \rfloor }& { \frac{1}{2} (k-2)(k-5)} & {\frac{1}{2} (k-1)(k-2)}  \\
          S_{k} &   \chi_{k,(2)}    &  1 & { \frac{1}{2} (k-2)(k-5)} + 1& {\frac{1}{2} (k-1)(k-2)} -1 \\
  \end{array}.
  }
\]
Let $X_{k,m}$ be the $S_k$-set consisting of $m$-tuples of distinct elements of $\{1,\dots,k\}$.  Then its permutation
character $\phi_{k,m}$ decomposes into $\chi_{k,\mu}$ with $\mu$ a partition of an integer $\leq m$.  These
formulas are uniform in $k$, as in 
\[
\phi_{k,2} = \chi_{k,(1,1)} + \chi_{k,(2)} + 2 \chi_{k,(1)} + \chi_{k,()}.
\]
For $\mu$ running over partitions of a large integer $m$, the characters $\chi_{k,\mu}$ can be
reasonably regarded as  quite far from permutation characters, and they thus serve as a better 
illustration of Corollary~\ref{limitcor}.   The sequences $(S_k,\chi_{k,\mu})$ satisfy Hypothesis A but not B, because
$n_k$ and $\widehat{\chi}_k$ grow polynomially as $k^m$, while $\widecheck{\chi}_k$ grows
polynomially with degree $<m$.  

\subsubsection{The extra-special group $2_\epsilon^{1+2k}$ and its degree $2^k$ character}  
\label{esk}
Fix $\epsilon \in \{+,-\}$.  
Let $G_k$ be the extra-special $2$-group of type $\epsilon$ 
and order $2^{1+2k}$, so that $2^{1+2}_+$ and $2^{1+2}_{-}$ 
are the dihedral and quaternion groups respectively.   These
groups each have exactly one irreducible character of degree
larger than $1$, this degree being $2^k$.
There are just three character values, $-2^k$, $0$, and $2^k$.   For these
two sequences, Corollary~\ref{limitcor} again applies, 
but now only through Hypothesis B.  

\subsubsection{The Weyl group $2^k.S_k$ and its degree $k$ reflection character}  The
Weyl group $W(B_k) \cong 2^k.S_k$ of signed permutation 
matrices comes with its defining degree $k$ character $\chi_k$. 
Here, as indicated by the fourth line of Table~\ref{fourasymp},
neither hypothesis of Corollary~\ref{limitcor} applies.

However the conclusion \eqref{limlowerbound} of Corollary~\ref{limitcor} 
continues to hold as follows.   Relate the character $\chi_k$ in question to the two
standard permutation characters of $2^k.S_k$ via
$\phi_{2k} = \phi_k + \chi_k$.
For a given $2^k.S_k$ field,
$D_{\Phi_{2k}}=D_{\Phi_k}D_{\Chi_k}$.  But, since $\Phi_k$
corresponds to an index $2$ subfield of the degree $2k$ number field
for $\Phi_{2k}$, we have $D_{\Phi_k}^2\mid D_{\Phi_{2k}}$.  Combining
these we get $D_{\Phi_k} \mid D_{\Chi_k}$ and hence $\delta_{\Phi_k} < \delta_{\Chi_k}$.  
So \eqref{lowfield} implies \eqref{limlowerbound}.  

\subsection{Concluding speculation} \label{speculation}
As we have illustrated in \S\ref{pglk}--\ref{esk}, both Hypothesis A and Hypothesis B
are quite broad.  This breadth, together with the fact that the conclusion \eqref{limlowerbound} still holds for our last sequence, raises
the question of whether \eqref{limlowerbound} can be formulated more universally.  While the evidence is far from definitive,
we expect a positive answer.  Thus we expect that the first accumulation point of the numbers $\delta_1(G,\chi)$ 
is at least $\Omega$, where $(G,\chi)$ runs over all types with $\chi$ irreducible.  
Phrased differently, we expect that the first accumulation point of the root 
conductors of all irreducible Artin $L$-functions is at least $\Omega$.

\bibliographystyle{amsalpha}
\bibliography{jr}

\end{document}

%% file: foo.tex
\begin{table}[htbp] \centering \caption{\label{tablelabel1}Artin $L$-functions with small conductor from groups in degrees $2$, $3$, $4$, and $5$}
\begin{tabular}{l@{\;}r@{\;}c@{\;}c@{\;}|@{\;}r@{\;}r@{\;}c@{}r@{\;}|@{\;}c@{\;}r@{\;}r}
$G$& $n_1$ & $z$ & $\Imnaught$ & \multicolumn{1}{c}{$\mathfrak{d}$} & \multicolumn{1}{c}{$\delta_1$} & $\Delta_1$ & pos'n & $\beta$ & $B^\beta$ & \# \\
\hline
$C_2$ & 2 &  TW  & 2, 2 &  & 1.73 & $3^*$& &  & 100 & 6086\\
2T1 & $1$ &  & $[-1, -1]$ & $2.97_{\ell}$ & 3.00 & $3$ & 1 &  $2.00_{\twnull}$ & 10000.00 & 6086\\
\hline
$C_3$ & 3 &  TW  & 2, 2 &  & 3.66 & $7^*$& &  & 500 & 1772\\
3T1 & $1$ & $\sqrt{-3}$ & $[-1, -1]$ & $6.93_{\ell}$ & 7.00 & $7$ & 1 &  $1.50_{\twnull}$ & 11180.34 & 1772\\
\hline
$S_3$ & 6 &  TW  & 3, 4 &  & 4.80 & $23^*$& &  & 250 & 24484\\
3T2 & $2$ &  & $[-1, 0]$ & $4.74_{\ell}$ & 4.80 & $23$ & 1 &  $1.00_{\twnull}$ & 250.00 & 13329\\
\hline
$C_4$ & 4 &  TW  & 3, 4 &  & 3.34 & $5^*$& &  & 150 & 2668\\
4T1 & $1$ & $i$ & $[-2, 0]$ & $4.96_{S}$ & 5.00 & $5$ & 1 &  $1.33_\twgood$ & 796.99 & 489\\
\hline
$D_4$ & 8 &  TW  & 5, 10 &  & 6.03 & $3^*7^*$& &  & 150 & 31742\\
4T3 & $2$ &  & $[-2, 0]$ & $5.74_{q}$ & 6.24 & $3\!\cdot\!  13$ & 2 &  $1.00_{\twnull}$ & 150.00 & 9868\\
\hline
$A_4$ & 12 &  TW  & 3, 4 &  & 10.35 & $2^*7^*$& &  & 150 & 846\\
4T4 & $3$ &  & $[-1, 0]$ & $7.60_{q}$ & 14.64 & $2^{6} 7^{2}$ & 1 &  $1.00_{\twnull}$ & 150.00 & 270\\
\hline
$S_4$ & 24 &  TW  & 5, 12 &  & 13.56 & $2^*11^*$& &  & 150 & 14587\\
6T8 & $3$ &  & $[-1, 1]$ & $8.62_{G}$ & 11.30 & $2^{2} 19^{2}$ & 4 &  $0.89_\twgood$ & 85.96 & 779\\
4T5 & $3$ &  & $[-1, 1]$ & $5.49_{p}$ & 6.12 & $229$ & 9 &  $0.67_{\twnull}$ & 28.23 & 1603\\
\hline
$C_5$ & 5 &  TW  & 2, 2 &  & 6.81 & $11^*$& &  & 200 & 49\\
5T1 & $1$ & $\zeta_{5}$ & $[-1, -1]$ & $10.67_{\ell}$ & 11.00 & $11$ & 1 &  $1.25_{\twnull}$ & 752.12 & 49\\
\hline
$D_5$ & 10 &  TW  & 3, 4 &  & 6.86 & $47^*$& &  & 200 & 3622\\
5T2 & $2$ & $\sqrt{5}$ & $[-1, 0]$ & $6.73_{q}$ & 6.86 & $47$ & 1 &  $1.00_{\twnull}$ & 200.00 & 3219\\
\hline
$F_5$ & 20 &  TW  & 4, 8 &  & 11.08 & $2^*5^*$& &  & 200 & 3010\\
5T3 & $4$ &  & $[-1, 0]$ & $10.28_{q}$ & 13.69 & $2^{4} 13^{3}$ & 2 &  $1.00_{\twnull}$ & 200.00 & 2066\\
\hline
$A_5$ & 60 &  TW  & 4, 8 &  & 18.70 & $2^*17^*$& &  & 85 & 473\\
5T4 & $4$ &  & $[-1, 1]$ & $8.18_{g}$ & 11.66 & $2^{6} 17^{2}$ & 1 &  $0.75_{\twnull}$ & 27.99 & 46\\
6T12 & $5$ &  & $[-1, 1]$ & $10.18_{p}$ & 12.35 & $2^{6} 67^{2}$ & 3 &  $0.80_{\twnull}$ & 34.96 & 216\\
12T33 & $3$ & $\sqrt{5}$ & $[-2, 1]$ & $10.34_{g}$ & 26.45 & $2^{6} 17^{2}$ & 1 &  $0.83_{\twnull}$ & 40.54 & 18\\
\hline
$S_5$ & 120 &  TW  & 7, 40 &  & 24.18 & $2^*3^*5^*$& &  & 85 & 2080\\
5T5 & $4$ &  & $[-1, 2]$ & $6.28_{\ell}$ & 6.33 & $1609$ & 103 &  $0.50_{\twnull}$ & 9.22 & 45\\
10T12 & $4$ &  & $[-2, 1]$ & $10.28_{V}$ & 18.72 & $5^{2} 17^{3}$ & 21 &  $0.75_{\twnull}$ & 27.99 & 15\\
10T13 & $5$ &  & $[-1, 1]$ & $12.13_{V}$ & 16.27 & $2^{4} 3^{2} 89^{2}$ & 6 &  $0.80_{\twnull}$ & 34.96 & 592\\
6T14 & $5$ &  & $[-1, 1]$ & $11.09_{g}$ & 17.78 & $2^{6} 3^{4} 7^{3}$ & 14 &  $0.80_{\twnull}$ & 34.96 & 186\\
20T35 & $6$ &  & $[-2, 1]$ & $12.26_{g}$ & 18.18 & $2^{4} 3^{3} 17^{4}$ & 12 &  $0.83_{\twnull}$ & 40.54 & 110\\
\end{tabular}
\end{table}
\begin{table}[htbp] \centering \caption{\label{tablelabel2}Artin $L$-functions of small conductor from sextic groups}
\begin{tabular}{l@{\;}r@{\;}c@{\;}c@{\;}|@{\;}r@{\;}r@{\;}c@{}r@{\;}|@{\;}c@{\;}r@{\;}r}
$G$& $n_1$ & $z$ & $\Imnaught$ & \multicolumn{1}{c}{$\mathfrak{d}$} & \multicolumn{1}{c}{$\delta_1$} & $\Delta_1$ & pos'n & $\beta$ & $B^\beta$ & \# \\
\hline
$C_6$ & 6 &  TW  & 4, 6 &  & 5.06 & $7^*$& &  & 200 & 9609\\
6T1 & $1$ & $\sqrt{-3}$ & $[-2, 1]$ & $6.93_{P}$ & 7.00 & $7$ & 1 &  $1.20_\twgood$ & 577.08 & 617\\
\hline
$D_6$ & 12 &  TW  & 6, 14 &  & 8.06 & $3^*5^*$& &  & 150 & 46197\\
6T3 & $2$ &  & $[-2, 1]$ & $7.60_{G}$ & 9.33 & $3\!\cdot\!  29$ & 6 &  $1.00_\twgood$ & 150.00 & 10242\\
\hline
$S_3C_3$ & 18 &  & 6, 17 &  & 10.06 & $2^*3^*7^*$& &  & 200 & 9420\\
6T5 & $2$ & $\sqrt{-3}$ & $[-2, 1]$ & $5.69_{q*}$ & 7.21 & $2^{2} 13$ & 4 &  $0.75_{\twnull}$ & 53.18 & 503\\
\hline
$A_4C_2$ & 24 &  & 6, 16 &  & 12.31 & $2^*7^*$& &  & 150 & 6676\\
6T6 & $3$ &  & $[-3, 1]$ & $7.60_{p}$ & 8.60 & $7^{2} 13$ & 3 &  $0.67_{\twnull}$ & 28.23 & 98\\
\hline
$S_3^2$ & 36 &  & 9, 69 &  & 15.53 & $2^*19^*$& &  & 200 & 45117\\
6T9 & $4$ &  & $[-2, 1]$ & $7.98_{q*}$ & 14.83 & $2^{4} 5^{2} 11^{2}$ & 27 &  $0.75_{\twnull}$ & 53.18 & 824\\
\hline
$C_3^2{\rtimes}C_4$ & 36 &  TW  & 5, 16 &  & 23.57 & $3^*5^*$& &  & 150 & 331\\
6T10 & $4^{\rlap{\scriptsize{2}}}$ &  & $[-2, 1]$ & $7.98_{q*}$ & 17.80 & $2^{11} 7^{2}$ & 2 &  $0.75_{\twnull}$ & 42.86 & 33\\
\hline
$S_4C_2$ & 48 &  & 10, 96 &  & 16.13 & $2^*23^*$& &  & 150 & 70926\\
6T11 & $3^{\rlap{\scriptsize{2}}}$ &  & $[-3, 1]$ & $6.14_{g}$ & 6.92 & $2^{2} 83$ & 7 &  $0.67_{\twnull}$ & 28.23 & 3694\\
\hline
$C_3^2{\rtimes}D_4$ & 72 &  TW  & 9, 105 &  & 21.76 & $3^*11^*$& &  & 150 & 8536\\
6T13 & $4^{\rlap{\scriptsize{2}}}$ &  & $[-2, 2]$ & $7.60_{p}$ & 7.90 & $3^{2} 433$ & 52 &  $0.50_{\twnull}$ & 12.25 & 41\\
12T36 & $4^{\rlap{\scriptsize{2}}}$ &  & $[-2, 1]$ & $11.29_{P}$ & 23.36 & $3^{5} 5^{2} 7^{2}$ & 18 &  $0.75_{\twnull}$ & 42.86 & 106\\
\hline
$A_6$ & 360 &  TW  & 6, 28 &  & 31.66 & $2^*3^*$& &  & 60 & 26\\
6T15 & $5^{\rlap{\scriptsize{2}}}$ &  & $[-1, 2]$ & $7.71_{\ell}$ & 12.35 & $2^{6} 67^{2}$ & 8 &  $0.60_{\twnull}$ & 11.67 & 0\\
10T26 & $9$ &  & $[-1, 1]$ & $17.69_{g}$ & 28.20 & $2^{18} 3^{16}$ & 1 &  $0.89_{\twnull}$ & 38.07 & 7\\
30T88 & $10$ &  & $[-2, 1]$ & $18.34_{g}$ & 30.61 & $2^{24} 3^{16}$ & 1 &  $0.90_{\twnull}$ & 39.84 & 4\\
36T555 & $8$ & $\sqrt{5}$ & $[-2, 1]$ & $20.70_{g}$ & 42.81 & $2^{18} 3^{16}$ & 1 &  $0.94_{\twnull}$ & 46.45 & 3\\
\hline
$S_6$ & 720 &  & 11, 596 &  & 33.50 & $2^*3^*5^*$& &  & 60 & 99\\
12T183 & $5^{\rlap{\scriptsize{2}}}$ &  & $[-3, 2]$ & $8.21_{v}$ & 11.53 & $11^{2} 41^{2}$ & 6 &  $0.60_{\twnull}$ & 11.67 & 1\\
6T16 & $5^{\rlap{\scriptsize{2}}}$ &  & $[-1, 3]$ & $6.23_{\ell}$ & 6.82 & $14731$ & 53 &  $0.40_{\twnull}$ & 5.14 & 0\\
10T32 & $9$ &  & $[-1, 3]$ & $10.77_{v}$ & \textit{16.60} & $2^{15} 11^{3} 13^{3}$ & 74 &  $0.67_{\twnull}$ & 15.33 & 0\\
20T145 & $9$ &  & $[-3, 1]$ & $19.33_{g}$ & 31.25 & $2^{6} 5^{6} 73^{4}$ & 16 &  $0.89_{\twnull}$ & 38.07 & 4\\
30T176 & $10^{\rlap{\scriptsize{2}}}$ &  & $[-2, 2]$ & $16.88_{v}$ & 24.22 & $11^{4} 41^{6}$ & 6 &  $0.80_{\twnull}$ & 26.46 & 1\\
36T1252 & $16$ &  & $[-2, 1]$ & $22.73_{g}$ & 35.46 & $2^{36} 3^{8} 7^{12}$ & 5 &  $0.94_{\twnull}$ & 46.45 & 11\\
\end{tabular}
\end{table}
\begin{table}[htbp] \centering \caption{\label{tablelabel3}Artin $L$-functions of small conductor from septic groups}
\begin{tabular}{l@{\;}r@{\;}c@{\;}c@{\;}|@{\;}r@{\;}r@{\;}c@{}r@{\;}|@{\;}c@{\;}r@{\;}r}
$G$& $n_1$ & $z$ & $\Imnaught$ & \multicolumn{1}{c}{$\mathfrak{d}$} & \multicolumn{1}{c}{$\delta_1$} & $\Delta_1$ & pos'n & $\beta$ & $B^\beta$ & \# \\
\hline
$C_7$ & 7 &  TW  & 2, 2 &  & 17.93 & $29^*$& &  & 200 & 15\\
7T1 & $1$ & $\zeta_{7}$ & $[-1, -1]$ & $14.03_{\ell}$ & 29.00 & $29$ & 1 &  $1.17_{\twnull}$ & 483.65 & 15\\
\hline
$D_7$ & 14 &  TW  & 3, 4 &  & 8.43 & $71^*$& &  & 200 & 2078\\
7T2 & $2$ & $\zeta_7^+$ & $[-1, 0]$ & $8.38_{q}$ & 8.43 & $71$ & 1 &  $1.00_{\twnull}$ & 200.00 & 1948\\
\hline
$C_7{\rtimes}C_3$ & 21 &  TW  & 3, 4 &  & 31.64 & $2^*73^*$& &  & 100 & 11\\
7T3 & $3$ & $\sqrt{-7}$ & $[-1, 0]$ & $25.50_{q}$ & 34.93 & $2^{3} 73^{2}$ & 1 &  $1.00_{\twnull}$ & 100.00 & 8\\
\hline
$F_7$ & 42 &  TW  & 5, 12 &  & 15.99 & $2^*7^*$& &  & 75 & 342\\
7T4 & $6$ &  & $[-1, 0]$ & $14.47_{q}$ & 18.34 & $11^{3} 13^{4}$ & 2 &  $1.00_{\twnull}$ & 75.00 & 287\\
\hline
$\textrm{GL}_3(2)$ & 168 &  TW  & 5, 14 &  & 32.25 & $2^*3^*11^*$& &  & 45 & 19\\
42T37 & $3$ & $\sqrt{-7}$ & $[-2, 2]$ & $15.55_{G}$ & 26.06 & $7^{2} 19^{2}$ & 7 &  $0.89_\twgood$ & 29.48 & 1\\
7T5 & $6$ &  & $[-1, 2]$ & $9.36_{p}$ & 11.23 & $13^{2} 109^{2}$ & 4 &  $0.67_{\twnull}$ & 12.65 & 1\\
8T37 & $7$ &  & $[-1, 1]$ & $14.10_{g}$ & 32.44 & $3^{8} 7^{8}$ & 11 &  $0.86_{\twnull}$ & 26.12 & 0\\
21T14 & $8$ &  & $[-1, 1]$ & $14.90_{g}$ & 23.16 & $2^{6} 3^{6} 11^{6}$ & 1 &  $0.88_{\twnull}$ & 27.96 & 1\\
\hline
$A_7$ & 2520 &  & 8, 115 &  & 39.52 & $2^*3^*7^*$& &  & 45 & 1\\
7T6 & $6$ &  & $[-1, 3]$ & $9.13_{\ell}$ & 12.54 & $3^{6} 73^{2}$ & 26 &  $0.50_{\twnull}$ & 6.71 & 0\\
15T47 & $14$ &  & $[-1, 2]$ & $19.39_{g}$ & \textit{36.05} & $3^{24} 53^{6}$ & 4 &  $0.86_{\twnull}$ & 26.12 & 0\\
21T33 & $14$ &  & $[-1, 2]$ & $19.39_{g}$ & \textit{31.07} & $3^{18} 17^{10}$ & 2 &  $0.86_{\twnull}$ & 26.12 & 0\\
42T294 & $15$ &  & $[-1, 3]$ & $18.18_{v}$ & \textit{35.73} & $2^{12} 3^{20} 7^{12}$ & 1 &  $0.80_{\twnull}$ & 21.02 & 0\\
70 & $10$ & $\sqrt{-7}$ & $[-4, 2]$ & $22.49_{g}$ & \textit{41.21} & $2^{9} 3^{14} 7^{8}$ & 1 &  $0.90_{\twnull}$ & 30.75 & 0\\
42T299 & $21$ &  & $[-3, 1]$ & $26.95_{g}$ & \textit{38.33} & $2^{18} 3^{30} 7^{16}$ & 1 &  $0.95_{\twnull}$ & 37.54 & 0\\
70 & $35$ &  & $[-1, 1]$ & $\mathit{28.79_{g}}$ & \textit{41.28} & $2^{30} 3^{50} 7^{28}$ & 1 &  $0.97_\twbad$ & 40.36 & 0\\
\hline
$S_7$ & 5040 &  & 15, -- &  & 40.49 & $2^*3^*5^*$& &  & 35 & 0\\
7T7 & $6$ &  & $[-1, 4]$ & $7.50_{\ell}$ & 7.55 & $184607$ &  &
 $0.33_{\twnull}$ & 3.27 & 0\\
14T46 & $6$ &  & $[-4, 3]$ & $\mathit{7.66_{p}}$ & \textit{17.02} & $2^{2} 7^{5} 19^{2}$ & 194 &  $0.50_{\twnull}$ & 5.92 & 0\\
30T565 & $14$ &  & $[-2, 4]$ & $16.32_{p}$ & \textit{26.02} & $2^{20} 53^{8}$ & 2 &  $0.71_{\twnull}$ & 12.67 & 0\\
30T565 & $14$ &  & $[-4, 2]$ & $20.24_{g}$ & \textit{30.98} & $2^{14} 71^{9}$ & 46 &  $0.86_{\twnull}$ & 21.06 & 0\\
42T413 & $14$ &  & $[-6, 2]$ & $20.24_{g}$ & \textit{38.27} & $2^{20} 3^{12} 11^{10}$ & 6 &  $0.86_{\twnull}$ & 21.06 & 0\\
21T38 & $14$ &  & $[-1, 6]$ & $13.12_{p}$ & \textit{22.02} & $2^{24} 3^{12} 29^{4}$ & 170 &  $0.57_{\twnull}$ & 7.63 & 0\\
42T412 & $15$ &  & $[-3, 5]$ & $16.96_{p}$ & \textit{32.90} & $3^{12} 5^{5} 11^{13}$ & 24 &  $0.67_{\twnull}$ & 10.70 & 0\\
42T411 & $15$ &  & $[-5, 3]$ & $16.56_{g}$ & \textit{29.92} & $2^{30} 3^{12} 17^{6}$ & 3 &  $0.80_{\twnull}$ & 17.19 & 0\\
70 & $20$ &  & $[-4, 2]$ & $23.53_{g}$ & \textit{35.18} & $2^{34} 53^{12}$ & 2 &  $0.90_{\twnull}$ & 24.53 & 0\\
42T418 & $21$ &  & $[-3, 3]$ & $20.24_{g}$ & \textit{33.42} & $2^{41} 3^{18} 17^{9}$ & 3 &  $0.86_{\twnull}$ & 21.06 & 0\\
84 & $21$ &  & $[-3, 1]$ & $28.27_{g}$ & \textit{39.59} & $2^{38} 3^{18} 7^{16}$ & 4 &  $0.95_{\twnull}$ & 29.55 & 0\\
70 & $35$ &  & $[-1, 5]$ & $\mathit{25.92_{p}}$ & \textit{40.71} & $2^{61} 3^{30} 7^{28}$ & 4 &  $0.86_{\twnull}$ & 21.06 & 0\\
126 & $35$ &  & $[-5, 1]$ & $\mathit{30.23_{g}}$ & \textit{43.26} & $2^{54} 3^{42} 5^{30}$ &  &
 $0.97_\twbad$ & 31.62 & 0\\
\end{tabular}
\end{table}
\begin{table}[htbp] \centering \caption{\label{tablelabel4}Artin $L$-functions of small conductor from octic groups}
\begin{tabular}{l@{\;}r@{\;}c@{\;}c@{\;}|@{\;}r@{\;}r@{\;}c@{}r@{\;}|@{\;}c@{\;}r@{\;}r}
$G$& $n_1$ & $z$ & $\Imnaught$ & \multicolumn{1}{c}{$\mathfrak{d}$} & \multicolumn{1}{c}{$\delta_1$} & $\Delta_1$ & pos'n & $\beta$ & $B^\beta$ & \# \\
\hline
$C_8$ & 8 &  TW  & 4, 8 &  & 11.93 & $17^*$& &  & 125 & 198\\
8T1 & $1$ & $\zeta_{8}$ & $[-4, 0]$ & $8.84_{S}$ & 17.00 & $17$ & 1 &  $1.14_\twgood$ & 249.15 & 41\\
\hline
$Q_8$ & 8 &  TW  & 5, 10 &  & 18.24 & $2^*3^*$& &  & 100 & 72\\
8T5 & $2$ &  & $[-2, 0]$ & $26.29_{S}$ & 48.00 & $2^{8} 3^{2}$ & 2 &  $1.33_\twgood$ & 464.16 & 41\\
\hline
$D_8$ & 16 &  TW  & 6, 20 &  & 9.75 & $5^*19^*$& &  & 125 & 6049\\
8T6 & $2$ & $\sqrt{2}$ & $[-4, 0]$ & $9.07_{q}$ & 9.75 & $5\!\cdot\!  19$ & 1 &  $1.00_{\twnull}$ & 125.00 & 2296\\
\hline
$C_8{\rtimes}C_2$ & 16 &  & 7, 24 &  & 9.32 & $3^*5^*$& &  & 125 & 672\\
8T7 & $2$ & $i$ & $[-4, 0]$ & $9.07_{q}$ & 15.00 & $3^{2} 5^{2}$ & 1 &  $1.00_{\twnull}$ & 125.00 & 75\\
\hline
$QD_{16}$ & 16 &  & 6, 20 &  & 10.46 & $2^*3^*$& &  & 125 & 1664\\
8T8 & $2$ & $\sqrt{-2}$ & $[-4, 0]$ & $9.07_{q}$ & 16.97 & $2^{5} 3^{2}$ & 1 &  $1.00_{\twnull}$ & 125.00 & 155\\
\hline
$Q_8{\rtimes}C_2$ & 16 &  & 9, 32 &  & 9.80 & $2^*3^*$& &  & 100 & 3366\\
8T11 & $2$ & $i$ & $[-4, 0]$ & $9.07_{q}$ & 10.95 & $2^{3} 3\!\cdot\!  5$ & 3 &  $1.00_{\twnull}$ & 100.00 & 825\\
\hline
$\textrm{SL}_2(3)$ & 24 &  TW  & 5, 14 &  & 29.84 & $163^*$& &  & 250 & 681\\
24T7 & $2$ &  & $[-2, 1]$ & $65.51_{P}$ & 163.00 & $163^{2}$ & 1 &  $1.20_\twgood$ & 754.27 & 94\\
8T12 & $2$ & $\sqrt{-3}$ & $[-4, 1]$ & $8.09_{p}$ & 12.77 & $163$ & 1 &  $0.75_{\twnull}$ & 62.87 & 78\\
\hline
 & 32 &  & 11, 74 &  & 13.79 & $2^*5^*$& &  & 125 & 11886\\
8T15 & $4$ &  & $[-4, 0]$ & $12.92_{q}$ & 16.12 & $2^{4} 5^{2} 13^{2}$ & 4 &  $1.00_{\twnull}$ & 125.00 & 3464\\
\hline
 & 32 &  & 9, 58 &  & 13.56 & $5^*11^*$& &  & 125 & 766\\
8T16 & $4$ &  & $[-4, 0]$ & $12.92_{q}$ & 16.58 & $5^{4} 11^{2}$ & 1 &  $1.00_{\twnull}$ & 125.00 & 129\\
\hline
$C_4\wr C_2$ & 32 &  & 10, 90 &  & 13.37 & $2^*5^*$& &  & 125 & 2748\\
8T17 & $2^{\rlap{\scriptsize{2}}}$ & $i$ & $[-4, 2]$ & $\mathit{5.74_{p}}$ & 8.25 & $2^{2} 17$ & 6 &  $0.50_\twbad$ & 11.18 & 3\\
\hline
 & 32 &  & 9, 58 &  & 14.05 & $2^*$& &  & 125 & 2720\\
8T19 & $4$ &  & $[-4, 0]$ & $12.92_{q}$ & 19.03 & $2^{17}$ & 1 &  $1.00_{\twnull}$ & 125.00 & 1282\\
\hline
 & 32 &  & 17, 806 &  & 18.42 & $2^*3^*5^*$& &  & 100 & 3284\\
8T22 & $4$ &  & $[-4, 0]$ & $12.92_{q}$ & 20.49 & $2^{4} 3^{2} 5^{2} 7^{2}$ & 3 &  $1.00_{\twnull}$ & 100.00 & 1162\\
\hline
$\textrm{GL}_2(3)$ & 48 &  & 7, 41 &  & 16.52 & $2^*43^*$& &  & 100 & 2437\\
24T22 & $2$ & $\sqrt{-2}$ & $[-4, 2]$ & $\mathit{5.74_{v}}$ & \textit{16.82} & $283$ & 2 &  $0.50_\twbad$ & 10.00 & 0\\
8T23 & $4$ &  & $[-4, 1]$ & $9.07_{p}$ & 9.95 & $3^{4} 11^{2}$ & 4 &  $0.75_{\twnull}$ & 31.62 & 99\\
\hline
$C_2^3{\rtimes}C_7$ & 56 &  TW  & 3, 4 &  & 17.93 & $29^*$& &  & 200 & 28\\
8T25 & $7$ &  & $[-1, 0]$ & $16.10_{q}$ & 17.93 & $29^{6}$ & 1 &  $1.00_{\twnull}$ & 200.00 & 27\\
\hline
 & 64 &  & 16, -- &  & 20.37 & $2^*5^*$& &  & 125 & 10317\\
8T26 & $4^{\rlap{\scriptsize{2}}}$ &  & $[-4, 2]$ & $9.07_{p}$ & \textit{12.85} & $3^{2} 5^{2} 11^{2}$ & 7 &  $0.50_\twbad$ & 11.18 & 0\\
\hline
$C_2\wr C_4$ & 64 &  & 11, 206 &  & 19.44 & $2^*$& &  & 125 & 2482\\
8T27 & $4^{\rlap{\scriptsize{2}}}$ &  & $[-4, 2]$ & $9.07_{p}$ & 10.60 & $5^{3} 101$ & 19 &  $0.50_{\twnull}$ & 11.18 & 1\\
\hline
$C_2 \wr C_2^2$ & 64 &  & 16, -- &  & 19.41 & $2^*7^*$& &  & 125 & 11685\\
8T29 & $4^{\rlap{\scriptsize{2}}}$ &  & $[-4, 2]$ & $9.07_{p}$ & 10.13 & $2^{4} 3^{2} 73$ & 28 &  $0.50_{\twnull}$ & 11.18 & 1\\
\end{tabular}
\end{table}
\begin{table}[htbp] \centering \caption{\label{tablelabel5}Artin $L$-functions of small conductor from octic groups}
\begin{tabular}{l@{\;}r@{\;}c@{\;}c@{\;}|@{\;}r@{\;}r@{\;}c@{}r@{\;}|@{\;}c@{\;}r@{\;}r}
$G$& $n_1$ & $z$ & $\Imnaught$ & \multicolumn{1}{c}{$\mathfrak{d}$} & \multicolumn{1}{c}{$\delta_1$} & $\Delta_1$ & pos'n & $\beta$ & $B^\beta$ & \# \\
\hline
 & 64 &  & 11, 206 &  & 19.44 & $2^*$& &  & 125 & 1217\\
8T30 & $4^{\rlap{\scriptsize{2}}}$ &  & $[-4, 2]$ & $9.07_{p}$ & \textit{14.57} & $5^{3} 19^{2}$ & 3 &  $0.50_\twbad$ & 11.18 & 0\\
\hline
 & 96 &  & 9, 49 &  & 34.97 & $2^*5^*13^*$& &  & 250 & 5520\\
8T32 & $4$ &  & $[-4, 1]$ & $9.12_{g}$ & 22.80 & $2^{6} 5^{2} 13^{2}$ & 2 &  $0.75_{\twnull}$ & 62.87 & 180\\
24T97 & $4$ & $\sqrt{-3}$ & $[-8, 1]$ & $\mathit{13.19_{g}}$ & 43.30 & $2^{6} 5^{2} 13^{3}$ & 2 &  $0.88_\twbad$ & 125.37 & 112\\
\hline
 & 96 &  & 8, 44 &  & 30.01 & $2^*5^*7^*$& &  & 150 & 791\\
8T33 & $6^{\rlap{\scriptsize{2}}}$ &  & $[-2, 2]$ & $11.29_{p}$ & 25.14 & $5^{3} 7^{4} 29^{2}$ & 12 &  $0.67_{\twnull}$ & 28.23 & 3\\
\hline
 & 96 &  & 10, 92 &  & 27.28 & $2^*3^*31^*$& &  & 110 & 1915\\
8T34 & $6$ &  & $[-2, 2]$ & $11.29_{p}$ & 22.61 & $31^{3} 67^{2}$ & 64 &  $0.67_{\twnull}$ & 22.96 & 1\\
\hline
$C_2\wr D_4$ & 128 &  & 20, -- &  & 22.91 & $2^*3^*13^*$& &  & 125 & 14369\\
8T35 & $4^{\rlap{\scriptsize{4}}}$ &  & $[-4, 2]$ & $9.07_{p}$ & 9.45 & $5^{2} 11\!\cdot\!  29$ & 110 &  $0.50_{\twnull}$ & 11.18 & 9\\
\hline
$C_2^3{\rtimes}F_{21}$ & 168 &  & 5, 14 &  & 31.64 & $2^*73^*$& &  & 200 & 342\\
8T36 & $7$ &  & $[-1, 1]$ & $16.06_{p}$ & 21.03 & $2^{6} 73^{4}$ & 1 &  $0.86_{\twnull}$ & 93.82 & 120\\
24T283 & $7$ & $\sqrt{-3}$ & $[-2, 1]$ & $\mathit{20.23_{p}}$ & 38.55 & $2^{6} 7^{11}$ & 2 &  $0.93_\twbad$ & 136.98 & 81\\
\hline
$C_2\wr A_4$ & 192 &  & 12, 700 &  & 37.27 & $2^*5^*7^*$& &  & 250 & 13649\\
8T38 & $4^{\rlap{\scriptsize{2}}}$ &  & $[-4, 2]$ & $8.56_{v}$ & 15.20 & $2^{6} 7^{2} 17$ & 11 &  $0.50_{\twnull}$ & 15.81 & 1\\
24T288 & $4^{\rlap{\scriptsize{2}}}$ & $\sqrt{-3}$ & $[-8, 4]$ & $\mathit{10.84_{p}}$ & \textit{23.95} & $2^{6} 3\!\cdot\!  5\!\cdot\!  7^{3}$ & 66 &  $0.50_{\twnull}$ & 15.81 & 0\\
\hline
 & 192 &  & 13, 559 &  & 32.35 & $2^*23^*$& &  & 100 & 1193\\
8T39 & $4^{\rlap{\scriptsize{2}}}$ &  & $[-4, 2]$ & $5.74_{p}$ & 8.71 & $2^{4} 359$ & 49 &  $0.50_{\twnull}$ & 10.00 & 8\\
24T333 & $8$ &  & $[-8, 1]$ & $\mathit{15.28_{g}}$ & 39.94 & $2^{10} 43^{6}$ & 2 &  $0.88_\twbad$ & 56.23 & 16\\
\hline
 & 192 &  & 13, 559 &  & 29.71 & $2^*23^*$& &  & 100 & 2001\\
8T40 & $4^{\rlap{\scriptsize{2}}}$ &  & $[-4, 2]$ & $\mathit{5.74_{p}}$ & \textit{13.04} & $2^{2} 5^{2} 17^{2}$ & 9 &  $0.50_\twbad$ & 10.00 & 0\\
24T332 & $8$ &  & $[-8, 1]$ & $\mathit{15.28_{g}}$ & 29.71 & $2^{12} 23^{6}$ & 1 &  $0.88_\twbad$ & 56.23 & 47\\
\hline
 & 192 &  & 14, 1210 &  & 28.11 & $2^*11^*$& &  & 100 & 4723\\
12T108 & $6^{\rlap{\scriptsize{2}}}$ &  & $[-2, 2]$ & $\mathit{11.29_{p}}$ & 20.78 & $2^{12} 3^{9}$ & 5 &  $0.67_{\twnull}$ & 21.54 & 2\\
8T41 & $6^{\rlap{\scriptsize{2}}}$ &  & $[-2, 2]$ & $11.29_{p}$ & 13.01 & $5^{3} 197^{2}$ & 13 &  $0.67_{\twnull}$ & 21.54 & 40\\
\hline
$A_4\wr C_2$ & 288 &  & 10, 178 &  & 32.18 & $2^*37^*$& &  & 135 & 1362\\
8T42 & $6$ &  & $[-2, 3]$ & $11.29_{p}$ & 11.58 & $5^{3} 139^{2}$ & 76 &  $0.50_{\twnull}$ & 11.62 & 1\\
18T112 & $9$ &  & $[-3, 1]$ & $\mathit{17.10_{g}}$ & 35.11 & $2^{24} 13^{6}$ & 2 &  $0.89_{\twnull}$ & 78.28 & 66\\
12T128 & $9$ &  & $[-3, 3]$ & $13.59_{p}$ & 22.52 & $7^{6} 233^{3}$ & 56 &  $0.67_{\twnull}$ & 26.32 & 6\\
24T703 & $6$ & $\sqrt{-3}$ & $[-4, 4]$ & $\mathit{13.79_{p}}$ & \textit{32.18} & $2^{4} 37^{5}$ & 1 &  $0.67_{\twnull}$ & 26.32 & 0\\
\hline
$C_2 \wr S_4$ & 384 &  & 20, -- &  & 31.38 & $5^*197^*$& &  & 100 & 6400\\
8T44 & $4^{\rlap{\scriptsize{4}}}$ &  & $[-4, 2]$ & $5.74_{p}$ & 7.53 & $5\!\cdot\!  643$ & 391 &  $0.50_{\twnull}$ & 10.00 & 26\\
24T708 & $8^{\rlap{\scriptsize{2}}}$ &  & $[-8, 4]$ & $\mathit{13.19_{p}}$ & \textit{25.55} & $2^{12} 5^{2} 11^{6}$ & 4 &  $0.50_{\twnull}$ & 10.00 & 0\\
\end{tabular}
\end{table}
\begin{table}[htbp] \centering \caption{\label{tablelabel6}Artin $L$-functions of small conductor from octic groups}
\begin{tabular}{l@{\;}r@{\;}c@{\;}c@{\;}|@{\;}r@{\;}r@{\;}c@{}r@{\;}|@{\;}c@{\;}r@{\;}r}
$G$& $n_1$ & $z$ & $\Imnaught$ & \multicolumn{1}{c}{$\mathfrak{d}$} & \multicolumn{1}{c}{$\delta_1$} & $\Delta_1$ & pos'n & $\beta$ & $B^\beta$ & \# \\
\hline
 & 576 &  & 16, -- &  & 29.35 & $2^*3^*$& &  & 100 & 2664\\
12T161 & $6$ &  & $[-2, 3]$ & $\mathit{7.60_{p}}$ & \textit{19.04} & $2^{8} 3^{3} 83^{2}$ & 1179 &  $0.50_{\twnull}$ & 10.00 & 0\\
8T45 & $6$ &  & $[-2, 3]$ & $7.60_{p}$ & 15.48 & $2^{14} 29^{2}$ & 110 &  $0.50_{\twnull}$ & 10.00 & 0\\
18T179 & $9$ &  & $[-3, 1]$ & $18.84_{g}$ & 29.69 & $2^{12} 5^{6} 23^{4}$ & 16 &  $0.89_{\twnull}$ & 59.95 & 121\\
12T165 & $9$ &  & $[-3, 3]$ & $10.60_{p}$ & 17.20 & $2^{6} 19^{3} 67^{3}$ & 161 &  $0.67_{\twnull}$ & 21.54 & 12\\
18T185 & $9^{\rlap{\scriptsize{2}}}$ &  & $[-3, 3]$ & $\mathit{13.74_{p}}$ & \textit{22.69} & $2^{12} 3^{11} 13^{3}$ & 6 &  $0.67_{\twnull}$ & 21.54 & 0\\
24T1504 & $12$ &  & $[-4, 4]$ & $\mathit{16.40_{p}}$ & \textit{35.03} & $2^{8} 5^{6} 31^{8}$ & 51 &  $0.67_{\twnull}$ & 21.54 & 0\\
\hline
 & 576 &  & 11, 522 &  & 49.75 & $3^*5^*7^*$& &  & 100 & 153\\
8T46 & $6$ &  & $[-2, 3]$ & $9.66_{p}$ & 19.51 & $3^{6} 5^{4} 11^{2}$ & 42 &  $0.50_{\twnull}$ & 10.00 & 0\\
12T160 & $6$ &  & $[-2, 3]$ & $\mathit{7.60_{p}}$ & \textit{27.27} & $2^{23} 7^{2}$ & 11 &  $0.50_{\twnull}$ & 10.00 & 0\\
16T1030 & $9$ &  & $[-3, 1]$ & $18.84_{g}$ & 35.40 & $2^{22} 3^{6} 13^{4}$ & 6 &  $0.89_{\twnull}$ & 59.95 & 50\\
18T184 & $9$ &  & $[-3, 1]$ & $18.84_{g}$ & 35.50 & $2^{12} 5^{7} 23^{4}$ & 7 &  $0.89_{\twnull}$ & 59.95 & 32\\
24T1506 & $12$ &  & $[-4, 4]$ & $\mathit{16.40_{p}}$ & \textit{55.16} & $2^{20} 3^{18} 5^{9}$ & 4 &  $0.67_{\twnull}$ & 21.54 & 0\\
36T766 & $9$ & $i$ & $[-6, 2]$ & $\mathit{18.84_{g}}$ & 58.55 & $2^{36} 7^{6}$ & 3 &  $0.89_{\twnull}$ & 59.95 & 2\\
\hline
$S_4\wr C_2$ & 1152 &  & 20, -- &  & 35.05 & $2^*5^*41^*$& &  & 150 & 23694\\
12T200 & $6$ &  & $[-2, 4]$ & $\mathit{7.60_{p}}$ & \textit{15.62} & $5^{3} 11^{2} 31^{2}$ & 20668 &  $0.33_\twbad$ & 5.31 & 0\\
8T47 & $6$ &  & $[-2, 4]$ & $7.60_{p}$ & 10.51 & $2^{9} 2633$ & 20566 &  $0.33_{\twnull}$ & 5.31 & 0\\
12T201 & $6$ &  & $[-4, 3]$ & $\mathit{7.60_{p}}$ & \textit{19.19} & $3^{7} 151^{2}$ & 21 &  $0.50_{\twnull}$ & 12.25 & 0\\
12T202 & $6$ &  & $[-4, 3]$ & $\mathit{7.60_{p}}$ & \textit{16.51} & $3^{10} 7^{3}$ & 12 &  $0.50_{\twnull}$ & 12.25 & 0\\
18T272 & $9$ &  & $[-3, 3]$ & $\mathit{9.70_{p}}$ & 19.73 & $3^{6} 853^{3}$ & 450 &  $0.67_{\twnull}$ & 28.23 & 44\\
18T274 & $9$ &  & $[-3, 3]$ & $15.95_{p}$ & 27.07 & $2^{12} 5^{7} 29^{3}$ & 105 &  $0.67_{\twnull}$ & 28.23 & 3\\
18T273 & $9$ &  & $[-3, 3]$ & $\mathit{12.88_{p}}$ & \textit{30.86} & $2^{16} 3^{18}$ & 5 &  $0.67_\twbad$ & 28.23 & 0\\
16T1294 & $9$ &  & $[-3, 3]$ & $10.60_{p}$ & 13.16 & $43^{3} 53^{3}$ & 39 &  $0.67_{\twnull}$ & 28.23 & 295\\
36T1946 & $12$ &  & $[-4, 4]$ & $\mathit{14.77_{p}}$ & \textit{32.80} & $2^{10} 13^{7} 17^{6}$ & 16 &  $0.67_{\twnull}$ & 28.23 & 0\\
24T2821 & $12$ &  & $[-4, 4]$ & $\mathit{16.03_{p}}$ & 26.48 & $2^{16} 5^{6} 41^{5}$ & 1 &  $0.67_{\twnull}$ & 28.23 & 1\\
36T1758 & $18$ &  & $[-6, 2]$ & $\mathit{20.29_{g}}$ & 36.08 & $2^{24} 5^{9} 41^{9}$ & 1 &  $0.89_{\twnull}$ & 85.96 & 1222\\
\end{tabular}
\end{table}
\begin{table}[htbp] \centering \caption{\label{tablelabel7}Artin $L$-functions of small conductor from nonic groups}
\begin{tabular}{l@{\;}r@{\;}c@{\;}c@{\;}|@{\;}r@{\;}r@{\;}c@{}r@{\;}|@{\;}c@{\;}r@{\;}r}
$G$& $n_1$ & $z$ & $\Imnaught$ & \multicolumn{1}{c}{$\mathfrak{d}$} & \multicolumn{1}{c}{$\delta_1$} & $\Delta_1$ & pos'n & $\beta$ & $B^\beta$ & \# \\
\hline
$C_9$ & 9 &  TW  & 3, 4 &  & 13.70 & $19^*$& &  & 200 & 48\\
9T1 & $1$ & $\zeta_{9}$ & $[-3, 0]$ & $17.02_{Q}$ & 19.00 & $19$ & 1 &  $1.13_\twgood$ & 387.85 & 26\\
\hline
$D_9$ & 18 &  TW  & 4, 8 &  & 12.19 & $2^*59^*$& &  & 200 & 699\\
9T3 & $2$ & $\zeta_9^+$ & $[-3, 0]$ & $9.70_{q}$ & 14.11 & $199$ & 3 &  $1.00_{\twnull}$ & 200.00 & 638\\
\hline
$3^{1+2}_{-}$ & 27 &  & 6, 12 &  & 31.18 & $7^*13^*$& &  & 200 & 32\\
9T6 & $3$ & $\sqrt{-3}$ & $[-3, 0]$ & $30.32_{q}$ & 38.70 & $7^{3} 13^{2}$ & 1 &  $1.00_{\twnull}$ & 200.00 & 14\\
\hline
$3^{1+2}_{+}$ & 27 &  TW  & 6, 12 &  & 50.20 & $3^*19^*$& &  & 200 & 16\\
9T7 & $3$ & $\sqrt{-3}$ & $[-3, 0]$ & $30.32_{q}$ & 64.08 & $3^{6} 19^{2}$ & 1 &  $1.00_{\twnull}$ & 200.00 & 12\\
\hline
$3^{1+2}.2$ & 54 &  & 7, 34 &  & 17.01 & $2^*3^*$& &  & 200 & 981\\
9T10 & $6$ &  & $[-3, 0]$ & $15.90_{q}$ & 17.49 & $31^{5}$ & 2 &  $1.00_{\twnull}$ & 200.00 & 741\\
\hline
 & 54 &  & 7, 34 &  & 16.83 & $3^*7^*$& &  & 200 & 880\\
9T11 & $6$ &  & $[-3, 0]$ & $15.90_{q}$ & 19.01 & $3^{9} 7^{4}$ & 1 &  $1.00_{\twnull}$ & 200.00 & 805\\
\hline
 & 54 &  & 8, 42 &  & 16.72 & $2^*3^*5^*$& &  & 200 & 2637\\
9T12 & $3$ & $\sqrt{-3}$ & $[-3, 2]$ & $7.75_{p}$ & 10.71 & $2^{2} 307$ & 7 &  $0.67_{\twnull}$ & 34.20 & 256\\
18T24 & $3$ & $\sqrt{-3}$ & $[-3, 1]$ & $\mathit{10.03_{g}}$ & 20.08 & $2^{2} 3^{4} 5^{2}$ & 1 &  $0.83_\twbad$ & 82.70 & 77\\
\hline
$M_9$ & 72 &  & 6, 20 &  & 29.72 & $2^*3^*$& &  & 100 & 27\\
9T14 & $8$ &  & $[-1, 0]$ & $17.50_{q}$ & 31.59 & $2^{24} 3^{10}$ & 1 &  $1.00_{\twnull}$ & 100.00 & 26\\
\hline
$C_3^2{\rtimes}C_8$ & 72 &  TW  & 5, 16 &  & 25.41 & $2^*3^*$& &  & 100 & 19\\
9T15 & $8$ &  & $[-1, 0]$ & $17.50_{q}$ & 25.41 & $2^{31} 3^{4}$ & 1 &  $1.00_{\twnull}$ & 100.00 & 16\\
\hline
$C_3\wr C_3$ & 81 &  & 9, 59 &  & 75.41 & $3^*19^*$& &  & 500 & 131\\
9T17 & $3^{\rlap{\scriptsize{3}}}$ & $\sqrt{-3}$ & $[-3, 3]$ & $12.92_{q}$ & 30.14 & $7^{2} 13\!\cdot\!  43$ & 22 &  $0.50_{\twnull}$ & 22.36 & 0\\
\hline
$C_3^2{\rtimes}D_6$ & 108 &  & 11, 262 &  & 22.06 & $3^*23^*$& &  & 150 & 12002\\
18T55 & $6$ &  & $[-3, 1]$ & $\mathit{11.98_{g}}$ & 24.38 & $2^{8} 3^{8} 5^{3}$ & 4 &  $0.83_\twbad$ & 65.07 & 229\\
9T18 & $6$ &  & $[-3, 2]$ & $9.70_{p}$ & 13.46 & $3^{3} 7^{2} 67^{2}$ & 53 &  $0.67_{\twnull}$ & 28.23 & 216\\
\hline
$C_3^2{\rtimes}QD_{16}$ & 144 &  & 8, 62 &  & 23.41 & $3^*7^*$& &  & 100 & 488\\
9T19 & $8$ &  & $[-1, 2]$ & $\mathit{12.13_{v}}$ & 25.65 & $3^{13} 7^{6}$ & 1 &  $0.75_{\twnull}$ & 31.62 & 14\\
18T68 & $8$ &  & $[-2, 1]$ & $\mathit{14.44_{g}}$ & 25.65 & $3^{13} 7^{6}$ & 1 &  $0.88_\twbad$ & 56.23 & 48\\
\hline
$C_3\wr S_3$ & 162 &  & 13, 2004 &  & 29.89 & $3^*$& &  & 200 & 1617\\
9T20 & $3^{\rlap{\scriptsize{3}}}$ & $\sqrt{-3}$ & $[-3, 3]$ & $7.75_{p}$ & 11.17 & $7\!\cdot\!  199$ & 23 &  $0.50_{\twnull}$ & 14.14 & 20\\
18T86 & $3^{\rlap{\scriptsize{3}}}$ & $\sqrt{-3}$ & $[-3, 3]$ & $\mathit{8.23_{v}}$ & \textit{19.48} & $2^{2} 43^{2}$ & 5 &  $0.50_{\twnull}$ & 14.14 & 0\\
\hline
 & 162 &  & 11, 223 &  & 24.90 & $2^*3^*5^*$& &  & 100 & 597\\
9T21 & $6^{\rlap{\scriptsize{3}}}$ &  & $[-3, 3]$ & $9.70_{p}$ & 15.58 & $5^{2} 83^{3}$ & 2 &  $0.50_{\twnull}$ & 14.14 & 0\\
\hline
 & 162 &  & 10, 205 &  & 26.46 & $3^*$& &  & 100 & 180\\
9T22 & $6^{\rlap{\scriptsize{3}}}$ &  & $[-3, 3]$ & $9.70_{p}$ & 17.21 & $2^{6} 7^{4} 13^{2}$ & 6 &  $0.50_{\twnull}$ & 10.00 & 0\\
\hline
$C_3^2{\rtimes}\textrm{SL}_2(3)$ & 216 &  & 7, 44 &  & 49.57 & $349^*$& &  & 100 & 37\\
9T23 & $8$ &  & $[-1, 2]$ & $11.29_{p}$ & 23.39 & $547^{4}$ & 10 &  $0.75_{\twnull}$ & 31.62 & 3\\
24T569 & $8$ & $\sqrt{-3}$ & $[-2, 1]$ & $18.99_{g}$ & 38.84 & $349^{5}$ & 1 &  $0.94_{\twnull}$ & 74.99 & 20\\
\hline
 & 324 &  & 17, -- &  & 30.64 & $2^*3^*11^*$& &  & 100 & 1816\\
18T129 & $6^{\rlap{\scriptsize{3}}}$ &  & $[-3, 3]$ & $\mathit{9.34_{p}}$ & \textit{22.88} & $2^{9} 23^{4}$ & 16 &  $0.50_{\twnull}$ & 14.14 & 0\\
9T24 & $6^{\rlap{\scriptsize{3}}}$ &  & $[-3, 3]$ & $9.70_{p}$ & 15.84 & $3^{7} 5^{2} 17^{2}$ & 399 &  $0.50_{\twnull}$ & 14.14 & 0\\
\end{tabular}
\end{table}
\begin{table}[htbp] \centering \caption{\label{tablelabel8}Artin $L$-functions of small conductor from nonic groups}
\begin{tabular}{l@{\;}r@{\;}c@{\;}c@{\;}|@{\;}r@{\;}r@{\;}c@{}r@{\;}|@{\;}c@{\;}r@{\;}r}
$G$& $n_1$ & $z$ & $\Imnaught$ & \multicolumn{1}{c}{$\mathfrak{d}$} & \multicolumn{1}{c}{$\delta_1$} & $\Delta_1$ & pos'n & $\beta$ & $B^\beta$ & \# \\
\hline
 & 324 &  & 9, 116 &  & 29.96 & $2^*3^*$& &  & 100 & 107\\
9T25 & $6$ &  & $[-3, 3]$ & $12.73_{p}$ & 22.25 & $2^{6} 3^{8} 17^{2}$ & 59 &  $0.50_{\twnull}$ & 10.00 & 0\\
18T141 & $6$ &  & $[-3, 3]$ & $\mathit{9.34_{p}}$ & \textit{30.81} & $3^{8} 19^{4}$ & 9 &  $0.50_{\twnull}$ & 10.00 & 0\\
12T133 & $4$ & $\sqrt{-3}$ & $[-4, 2]$ & $11.15_{g}$ & 19.34 & $2^{6} 3^{7}$ & 1 &  $0.75_{\twnull}$ & 31.62 & 10\\
12T132 & $4^{\rlap{\scriptsize{2}}}$ & $\sqrt{-3}$ & $[-4, 2]$ & $\mathit{11.15_{g}}$ & \textit{33.50} & $2^{6} 3^{9}$ & 1 &  $0.75_{\twnull}$ & 31.62 & 0\\
18T142 & $12$ &  & $[-3, 3]$ & $19.68_{p}$ & 30.57 & $2^{18} 3^{26}$ & 3 &  $0.75_{\twnull}$ & 31.62 & 2\\
\hline
$C_3^2{\rtimes}\textrm{GL}_2(3)$ & 432 &  & 10, 206 &  & 27.88 & $3^*11^*$& &  & 76 & 453\\
9T26 & $8$ &  & $[-1, 2]$ & $11.54_{g}$ & 17.59 & $2^{6} 523^{3}$ & 14 &  $0.75_{\twnull}$ & 25.74 & 17\\
18T157 & $8$ &  & $[-2, 2]$ & $\mathit{12.14_{p}}$ & 19.04 & $3^{7} 53^{4}$ & 16 &  $0.75_{\twnull}$ & 25.74 & 3\\
24T1334 & $16$ &  & $[-2, 1]$ & $21.27_{g}$ & 26.68 & $3^{26} 11^{10}$ & 1 &  $0.94_{\twnull}$ & 57.98 & 134\\
\hline
$S_3\wr C_3$ & 648 &  & 14, 3706 &  & 33.56 & $2^*5^*13^*$& &  & 150 & 1677\\
18T197 & $6^{\rlap{\scriptsize{2}}}$ &  & $[-3, 3]$ & $\mathit{9.70_{v}}$ & \textit{19.71} & $7^{4} 29^{3}$ & 1 &  $0.50_{\twnull}$ & 12.25 & 0\\
18T202 & $6$ &  & $[-4, 3]$ & $\mathit{9.70_{v}}$ & \textit{27.73} & $2^{6} 3^{9} 19^{2}$ & 49 &  $0.50_{\twnull}$ & 12.25 & 0\\
9T28 & $6$ &  & $[-3, 4]$ & $9.70_{p}$ & 12.20 & $3^{8} 503$ & 335 &  $0.33_{\twnull}$ & 5.31 & 0\\
12T176 & $8$ &  & $[-4, 2]$ & $12.05_{g}$ & 21.75 & $11^{4} 43^{4}$ & 4 &  $0.75_{\twnull}$ & 42.86 & 57\\
36T1102 & $12$ &  & $[-4, 3]$ & $\mathit{15.58_{v}}$ & 29.90 & $2^{6} 5^{10} 13^{8}$ & 1 &  $0.75_{\twnull}$ & 42.86 & 2\\
18T206 & $12$ &  & $[-3, 4]$ & $15.23_{v}$ & 21.99 & $2^{6} 7^{10} 29^{4}$ & 10 &  $0.67_{\twnull}$ & 28.23 & 8\\
24T1539 & $8$ & $\sqrt{-3}$ & $[-8, 4]$ & $\mathit{17.07_{v}}$ & \textit{45.71} & $2^{14} 3^{19}$ & 3 &  $0.75_{\twnull}$ & 42.86 & 0\\
\hline
 & 648 &  & 13, 2206 &  & 40.81 & $2^*3^*17^*$& &  & 200 & 838\\
9T29 & $6$ &  & $[-3, 3]$ & $12.73_{p}$ & 16.62 & $2^{8} 7^{2} 41^{2}$ & 31 &  $0.50_{\twnull}$ & 14.14 & 0\\
18T223 & $6$ &  & $[-3, 3]$ & $\mathit{9.70_{v}}$ & \textit{30.14} & $2^{4} 5^{2} 37^{4}$ & 71 &  $0.50_{\twnull}$ & 14.14 & 0\\
24T1527 & $4$ & $\sqrt{-3}$ & $[-4, 2]$ & $12.05_{g}$ & 32.34 & $2^{2} 3^{7} 5^{3}$ & 18 &  $0.75_{\twnull}$ & 53.18 & 43\\
12T175 & $4$ & $\sqrt{-3}$ & $[-4, 4]$ & $7.60_{p}$ & 9.23 & $11\!\cdot\!  659$ & 164 &  $0.50_{\twnull}$ & 14.14 & 14\\
36T1131 & $6$ & $\sqrt{-3}$ & $[-6, 6]$ & $\mathit{11.95_{v}}$ & \textit{36.04} & $2^{2} 3^{8} 17^{4}$ & 6 &  $0.50_{\twnull}$ & 14.14 & 0\\
36T1237 & $12$ &  & $[-3, 3]$ & $\mathit{17.78_{p}}$ & 44.72 & $2^{22} 3^{7} 17^{8}$ & 1 &  $0.75_{\twnull}$ & 53.18 & 6\\
18T219 & $12$ &  & $[-3, 3]$ & $\mathit{14.05_{v}}$ & 33.23 & $3^{17} 107^{5}$ & 5 &  $0.75_{\twnull}$ & 53.18 & 65\\
24T1540 & $8$ & $\sqrt{-3}$ & $[-8, 4]$ & $\mathit{17.07_{v}}$ & 49.37 & $2^{10} 3^{15} 7^{4}$ & 2 &  $0.75_{\twnull}$ & 53.18 & 1\\
\hline
 & 648 &  & 13, 1322 &  & 30.37 & $2^*269^*$& &  & 200 & 4001\\
9T30 & $6$ &  & $[-3, 3]$ & $9.70_{p}$ & 10.67 & $11^{2} 23^{3}$ & 3 &  $0.50_{\twnull}$ & 14.14 & 1\\
18T222 & $6$ &  & $[-3, 3]$ & $\mathit{9.70_{v}}$ & \textit{23.27} & $31^{3} 73^{2}$ & 57 &  $0.50_{\twnull}$ & 14.14 & 0\\
12T178 & $8$ &  & $[-4, 2]$ & $12.05_{g}$ & 18.27 & $2^{10} 59^{4}$ & 10 &  $0.75_{\twnull}$ & 53.18 & 327\\
12T177 & $8^{\rlap{\scriptsize{2}}}$ &  & $[-4, 2]$ & $\mathit{12.05_{g}}$ & 24.98 & $2^{10} 23^{6}$ & 22 &  $0.75_{\twnull}$ & 53.18 & 173\\
36T1121 & $6$ & $\sqrt{3}$ & $[-6, 6]$ & $\mathit{11.95_{v}}$ & \textit{31.61} & $2^{8} 7^{2} 43^{3}$ & 91 &  $0.50_{\twnull}$ & 14.14 & 0\\
36T1123 & $12$ &  & $[-3, 3]$ & $\mathit{17.78_{p}}$ & 28.87 & $2^{35} 5^{10}$ & 2 &  $0.75_{\twnull}$ & 53.18 & 71\\
18T218 & $12$ &  & $[-3, 3]$ & $14.05_{v}$ & 18.33 & $2^{10} 269^{5}$ & 1 &  $0.75_{\twnull}$ & 53.18 & 453\\
\hline
$S_3 \wr S_3$ & 1296 &  & 22, -- &  & 36.26 & $2^*3^*$& &  & 200 & 12152\\
18T320 & $6$ &  & $[-3, 4]$ & $\mathit{8.80_{p}}$ & \textit{14.80} & $5\!\cdot\!  23^{3} 173$ & 8562 &  $0.33_{\twnull}$ & 5.85 & 0\\
18T312 & $6$ &  & $[-4, 3]$ & $\mathit{8.79_{p}}$ & \textit{17.45} & $3^{5} 11^{2} 31^{2}$ & 343 &  $0.50_{\twnull}$ & 14.14 & 0\\
9T31 & $6$ &  & $[-3, 4]$ & $9.70_{p}$ & 10.38 & $31^{2} 1303$ & 10036 &  $0.33_{\twnull}$ & 5.85 & 0\\
18T303 & $6$ &  & $[-4, 3]$ & $\mathit{8.79_{p}}$ & \textit{18.34} & $5^{5} 23^{3}$ & 4 &  $0.50_{\twnull}$ & 14.14 & 0\\
12T213 & $8$ &  & $[-4, 4]$ & $11.29_{p}$ & 13.38 & $5^{2} 7^{4} 131^{2}$ & 397 &  $0.50_{\twnull}$ & 14.14 & 3\\
24T2895 & $8$ &  & $[-4, 2]$ & $\mathit{12.79_{g}}$ & 30.65 & $2^{8} 3^{4} 5^{6} 7^{4}$ & 77 &  $0.75_{\twnull}$ & 53.18 & 230\\
18T315 & $12$ &  & $[-3, 4]$ & $13.59_{p}$ & 22.81 & $2^{10} 23^{4} 37^{5}$ & 72 &  $0.67_{\twnull}$ & 34.20 & 105\\
36T2216 & $12$ &  & $[-3, 4]$ & $\mathit{13.79_{p}}$ & 29.13 & $2^{10} 3^{17} 41^{4}$ & 78 &  $0.67_{\twnull}$ & 34.20 & 11\\
36T2305 & $12$ &  & $[-4, 3]$ & $\mathit{15.14_{p}}$ & 31.73 & $2^{18} 331^{5}$ & 58 &  $0.75_{\twnull}$ & 53.18 & 151\\
36T2211 & $12$ &  & $[-6, 6]$ & $\mathit{14.64_{p}}$ & \textit{37.29} & $2^{16} 5^{8} 7^{10}$ & 55 &  $0.50_{\twnull}$ & 14.14 & 0\\
36T2214 & $12$ &  & $[-4, 3]$ & $\mathit{15.14_{p}}$ & 32.07 & $2^{22} 3^{24}$ & 11 &  $0.75_{\twnull}$ & 53.18 & 136\\
24T2912 & $16$ &  & $[-8, 4]$ & $\mathit{18.82_{p}}$ & 35.12 & $5^{12} 23^{12}$ & 4 &  $0.75_{\twnull}$ & 53.18 & 40\\
\end{tabular}
\end{table}